\documentclass[preprint,12pt]{elsarticle}

\usepackage{hyperref}

\usepackage{amssymb}
\usepackage{amsmath}
\usepackage{amsthm}

\usepackage{mathtools}
\usepackage{bbold}
\usepackage{bm}
\usepackage{IEEEtrantools}

\journal{Theoretical Computer Science}

\allowdisplaybreaks[1]                  
\numberwithin{equation}{section}
\renewcommand{\appendixname}{\!\!}

\newcommand{\R}{\mathbb{R}}  
\renewcommand{\P}{\mathbb{P}}  
\newcommand{\Z}{\mathbb{Z}}  

\newcommand{\U}{\mathbf{U}}  
\newcommand{\V}{\mathbf{V}}  
\newcommand{\X}{\mathbf{X}}  
\newcommand{\Y}{\mathbf{Y}}  
\newcommand{\s}{\mathbf{s}}  
\newcommand{\uu}{\mathbf{u}}  
\newcommand{\vv}{\mathbf{v}}  
\newcommand{\x}{\mathbf{x}}  
\newcommand{\y}{\mathbf{y}}  
\newcommand{\z}{\mathbf{z}}  
\newcommand{\g}{\mathbf{g}}  
\newcommand{\rr}{\mathbf{r}}  

\newcommand{\tb}{\tilde{b}}
\newcommand{\tc}{\tilde{c}}

\newcommand{\hF}{\widehat{F}}

\newcommand{\fb}{\mathfrak{b}}

\newcommand{\cS}{\mathcal{S}} 
\newcommand{\cL}{\mathcal{L}} 
\newcommand{\cP}{\mathcal{P}} 
\newcommand{\Pto}{\overset{\mathrm{P}}{\longrightarrow}}

\newcommand{\TV}{\mathrm{d_{TV}}}  
\newcommand{\rp}{\mathrm{p}}  
\newcommand{\eps}{\varepsilon}
\newcommand{\vp}{\varphi}

\newcommand{\iid}{i.i.d.}
\newcommand{\ie}{i.e.}

\newcommand\half{\tfrac12}

\newcommand{\ga}{\alpha} 
\newcommand{\gb}{\beta} 
\newcommand{\gl}{\lambda} 
\newcommand\gs{\sigma} 
\newcommand\gD{\Delta} 
\newcommand\gam{\gamma} 
\newcommand{\gz}{\zeta} 
\newcommand{\go}{\omega} 
\newcommand\hvp{\widehat{\vp}} 
\newcommand\hgs{\widehat{\bm{\gs}}} 

\newcommand\bzn{b^*_n}

\newcommand\bun{\underline{b}_n}
\newcommand\gbun{\underline{\gb}_n}
\newcommand\muun{\underline{\mu}_n}
\newcommand\gu{\underline{\g}}

\newcommand\bon{\overline{b}_n}
\newcommand\Non{\overline{N}_n}
\newcommand\muon{\overline{\mu}_n}
\newcommand{\Aon}{\overline{A}_n}
\newcommand{\Pion}{\overline{\Pi}_n}
\newcommand{\Mon}{\overline{M}_n}
\newcommand{\gov}{\overline{\g}}

\newcommand{\han}{\widehat{u}_n} 

\newcommand\glnm{\gl_{n, m}}

\newcommand{\refT}[1]{Theorem~\ref{#1}}

\newcommand{\refL}[1]{Lemma~\ref{#1}}
\newcommand{\refR}[1]{Remark~\ref{#1}}
\newcommand{\refS}[1]{Section~\ref{#1}}

\newcommand{\refP}[1]{Proposition~\ref{#1}}

\newcommand{\refConj}[1]{Conjecture~\ref{#1}}

\newcommand{\refApp}[1]{Appendix~\ref{#1}}

\newcommand\noqed{\renewcommand{\qed}{}} 

\newcommand\dd{\,\mathrm{d}}
\newcommand\ddx{\mathrm{d}}

\DeclareMathOperator{\E}{\mathbb{E}}  
\DeclareMathOperator{\Var}{Var}  
\DeclareMathOperator{\rL}{L}  
\DeclareMathOperator{\Po}{Po}
\DeclareMathOperator{\Gam}{Gamma}
\newcommand{\ignore}[1]{}

\theoremstyle{plain}
\newtheorem{theorem}{Theorem}[section]
\newtheorem{lemma}[theorem]{Lemma}
\newtheorem{proposition}[theorem]{Proposition}

\newtheorem{conj}[theorem]{Conjecture}

\theoremstyle{definition}

\newtheorem{remark}[theorem]{Remark}

\theoremstyle{remark}

\begin{document}

\begin{frontmatter}



\title{A New Fine-scale Berry--Esseen-type Gumbel-limit Theorem for Multivariate Maxima}

\author{James Allen Fill}
\affiliation{organization={Department of Applied Mathematics and Statistics, The Johns Hopkins University},
addressline={3400 N.~Charles Street},
city={Baltimore},
postcode={21218-2682},
state={MD},
country={USA}}


\begin{abstract}
For $d \geq 2$ and \iid\ $d$-dimensional observations $\X^{(1)}, \X^{(2)}, \ldots$ with independent Exponential$(1)$ coordinates, let $\vp_n$ denote the minimum $\ell^1$-norm among the maxima of $\{\X^{(1)}, \ldots, \X^{(n)}\}$.  (A \emph{maximum} from this set is an observation $\X^{(k)}$ with $1 \leq k \leq n$ such that $\X^{(k)} \not\prec \X^{(i)}$ for all $1 \leq i \leq n$, where $\x \prec \y$ means that $x_j < y_j$ for $1 \leq j \leq d$.)  Key roles in the study of multivariate Pareto records are played by $\vp_n$ and by the more easily handled maximum with the maximum $\ell^1$-norm.  Fill et al.\ proved~\cite[Theorem 1.11(a)]{FNS_2024} that
\[
\vp_n = \ln n - \ln \ln \ln n - \ln(d - 1) + O_\rp\!\left( \frac{1}{\ln \ln n} \right),
\]
where $Z_n = O_\rp(a_n)$ means that $Z_n / a_n$ is bounded in probability,
and conjectured~\cite[Remarks~1.13 and~3.3]{FNS_2024} that
\[
(\ln \ln n) \left(\vp_n - [\ln n - \ln \ln \ln n - \ln(d - 1)] \right)
\]
has a nondegenerate limiting distribution, suggesting that the limiting distribution might be that of $ - G$, where~$G$ has a Gumbel distribution with location $ - \frac{\ln[(d - 1)!]}{d - 1}$ and scale $\frac{1}{d - 1}$.  In the present paper we prove a Berry--Esseen-type theorem for this convergence in distribution, thereby establishing a very sharp result for $\vp_n$.
\end{abstract}



\begin{keyword}
Multivariate maxima \sep Gumbel distributions \sep Berry--Esseen-type theorem \sep Poisson approximation \sep 
Chen--Stein method \sep multivariate Pareto records


\MSC[2020] Primary:\ 60D05 \sep Secondary:\ 60F05

\end{keyword}

\end{frontmatter}



\newpage

\section{Introduction, background, assumptions, notation, and statement of main results}
\label{S:intro}

{\bf Some notation:\ }We 
define ${\bf 1}(A)$ to be $0$ or $1$ according as~$A$ is false or true.  We denote by $\cL(Y)$ the distribution, or law, of a random variable~$Y$.  For a positive integer~$n$, let $[n] := \{1, \ldots, n\}$.  We find it very useful to abbreviate the $k$th iterate of natural logarithm $\ln$ by $\rL_k$ and $\rL_1$ by $\rL$.  We fix a dimension $d \geq 2$, and $\x$ is then shorthand for the vector $(x_1, \ldots, x_d)$.  We write ${\bf 0}$ for a vector of $0$'s and $\x \prec \y$ (or $\y \succ \x$) to mean that 
$x_j < y_j$ for every $j \in [d]$.
 For simplicity of notation, we will write $\| \x \|$ (rather than $\| \x \|_{\ell^1}$ or $\| \x \|_1$) for the $\ell^1$-norm 
$\sum_{j = 1}^d |x_j|$ ($\ = \sum_{j = 1}^d x_j$ if $\x \succ {\bf 0}$) of~$\x$.  We take the \emph{probabiity simplex} 
$\cS_{d - 1}$ to be the set of all $\x \succ {\bf 0}$ with $\| \x \| = 1$.  All the results of this paper hold for any fixed dimension 
$d \geq 2$, and all asymptotic assertions are as $n \to \infty$ unless otherwise specified.
\smallskip

Univariate records have been quite thoroughly studied (see \cite{Arnold_1998} for a book-length treatment), but there is still fundamental knowledge to uncover about multivariate (Pareto) records, and specifically about (partial) \emph{maxima} [see the next paragraph for a definition of maximum; in the standard terminology of \cite{Fillboundary_2020}, a maximum at epoch~$n$ is a \emph{current} or \emph{remaining} (Pareto) \emph{record}], even under the simple category of models (call the category by the name~S) that observations are \iid\ (independent and identically distributed) and coordinates of each observation are independent and continuously distributed (with possibly different distributions for the different coordinates).

We assume throughout this paper that $\X^{(1)}, \X^{(2)}, \ldots$ is a sequence of \iid\ copies of a random vector 
$\X = (X_1, \ldots, X_d)$ taking values in $(0, \infty)^d$ and having independent Exponential$(1)$ coordinates.  Call this Model~E, a member of Category~S.  We are interested in the sequence, indexed by~$n$, of the set of maxima of $\{\X^{(1)}, \ldots, \X^{(n)}\}$ at epoch~$n$; a \emph{maximum} from this set is an observation $\X^{(k)}$ with 
$1 \leq k \leq n$ such that $\X^{(k)} \not\prec \X^{(i)}$ for all $1 \leq i \leq n$.
Results about maxima under Model~E are readily transformed to results under any specific model falling in Category~S. 

It is natural and important to study the $\ell^1$-norms of maxima under Model~E. 
One motivation for this study is that \cite[Theorem~1.4]{Fillboundary_2020} establishes the basic result that the conditional distribution of $\| \X^{(k)} \|$ given that $\X^{(k)}$ is a current record at epoch~$n$ doesn’t depend on 
$k \in [n]$ (so let us suppose now that $k = n$) and (after centering by $\ln n$) this ``generic maximum'' has a weak limit that is standard Gumbel (\ie,\ has distribution function $x \mapsto e^{- e^{-x}}$, $x \in \R$); moreover, as seen from \cite[Remark 2.1(b)]{Fillmulti_2023}, the proof of \cite[Theorem~1.4]{Fillboundary_2020} shows that 
$\|\X^{(n)}\|$ and $\U_n := \|\X^{(n)}\|^{-1} \X^{(n)}$ are independent, with $\U_n$ uniformly distributed on the probability simplex $\cS_{d - 1}$.  Another motivation is that the observation by Bai et al.~\cite[last paragraph of Section~1 and Section~3.2]{Bai_2005} ``that nearly all maxima occur in a thin strip sandwiched between [the] two parallel hyper-planes'' 
\begin{equation}
\label{Bai hyper}
\| \x \| = \rL n - \rL_3 n - \rL[4 (d - 1)] \quad \mbox{and} \quad \| \x \| = \rL n + 4 (d - 1) \rL_2 n.
\end{equation}
is one key to their proof of asymptotic normality for the number of maxima at epoch~$n$, a result that had been indicated by Baryshnikov~\cite{Baryshnikov_2000}.

Fill and Naiman~\cite{Fillboundary_2020} initiated, and Fill et al.\ \cite{FNS_2024} refined, the study of sharp localization of (separately) the maximum and minimum $\ell^1$-norms of maxima (and related quantities) at epoch~$n$ under Model~E.  See \cite[Theorem~1.8]{Fillboundary_2020} (or \cite[Theorem~1.6]{FNS_2024}), due essentially to the classical extreme value theory of Kiefer~\cite{Kiefer_1972}, for treatment of the maximum $\ell^1$-norm, denoted there as $F^+_n$.  In particular, the theorem establishes that
\begin{equation}
\label{F+n}
F^+_n - \big[ \rL n + (d - 1) \rL_2 n - \rL((d - 1)!) \big]
\end{equation}
converges in distribution to standard Gumbel.

This paper concerns the minimum $\ell^1$-norm, denoted $\hF^-_n$ in \cite{FNS_2024} and more simply $\vp_n$ here, of any maximum at epoch~$n$.  Fill et al.\ \cite{FNS_2024} applied the first- and second-moment methods to the count (for arbitrarily chosen $a \in \R$) of the number of maxima with $\ell^1$-norm smaller than 
\begin{equation}
\label{b0n}
\bzn \equiv \bzn(a) := \rL n - \rL_3 n - \rL(d - 1) + \frac{a}{\rL_2 n}
\end{equation} 
to prove~\cite[Theorem 1.11(a)]{FNS_2024} that
\[
\vp_n = \rL n - \rL_3 n - \rL(d - 1) + O_\rp\!\left( \frac{1}{\rL_2 n} \right),
\]
where $Z_n = O_\rp(a_n)$ means that $Z_n / a_n$ is bounded in probability,
and conjectured~\cite[Remarks~1.13 and~3.3]{FNS_2024} that
\begin{equation}
\label{phin normalized}
\vp_n^{\circ} := (\rL_2 n) \left(\vp_n - [\rL n - \rL_3 n - \rL(d - 1)] \right)
\end{equation}
has a nondegenerate limiting distribution, suggesting that the limiting distribution might be that of $ - G$, where~$G$ has a Gumbel distribution with location $ - \frac{\ln[(d - 1)!]}{d - 1}$ and scale $\frac{1}{d - 1}$.  In the present paper we prove this convergence in distribution, with a Berry--Esseen-type bound, thereby describing the asymptotic behavior of $\vp_n$ on a quite fine scale---in particular, finer than that for the much more easily studied $F^+_n$.

Here then is the main result of this paper.  Recall that if $G_0$ is distributed standard Gumbel (with distribution function $x \mapsto e^{- e^{-x}}$, $x \in \R$) and $m \in \R$ and $s \in (0, \infty)$, then $m + s\,G_0$ has a Gumbel distribution with location~$m$ and scale~$s$.  Recall also the definition of \emph{Kolmogorov distance} $d_{\mathrm{K}}$ between two probability measures $\mu_1$ and $\mu_2$ on~$(\R, \mbox{Borels})$ with respective distribution functions $F_1$ and $F_2$:
\[
d_{\mathrm{K}}(\mu_1, \mu_2) := \sup_{t \in \R} |F_1(t) - F_2(t)|.
\]

\begin{theorem}[main theorem:\ Gumbel limit]
\label{T:main}
Fix $d \geq 2$.  Let $\X^{(1)}, \X^{(2)}, \ldots$, $\X^{(n)}$ be independent random $d$-dimensional vectors each having independent {\rm Exponential}$(1)$ coordinates, let $\vp_n$ denote the minimum $\ell^1$-norm of any maximum of these vectors, and define $\vp_n^{\circ}$ as at~\eqref{phin normalized}.  Then $\vp_n^{\circ}$ converges in distribution to $- G$, where~$G$ has a Gumbel distribution with location $ - \frac{\rL[(d - 1)!]}{d - 1}$ and scale $\frac{1}{d - 1}$; moreover,
\begin{equation}
\label{BE}
d_{\mathrm{K}}(\cL(\vp_n^{\circ}), \cL(- G)) = O((\rL_2 n)^{-1 / 2} \rL_3 n).
\end{equation} 
\end{theorem}
\noindent
In other words, \eqref{BE} asserts precisely that
\begin{equation}
\label{BE again 1}
\left| \P(\vp_n^{\circ} > a) - \exp\left[ - \frac{e^{(d - 1) a}}{(d - 1)!} \right] \right| 
= O((\rL_2 n)^{-1 / 2} \rL_3 n)\mbox{\ uniformly in $a \in \R$},
\end{equation}
or equivalently that
\begin{equation}
\label{BE again 2}
\left| \P(\vp_n > \bzn(a)) - \exp\left[ - \frac{e^{(d - 1) a}}{(d - 1)!} \right] \right|
= O((\rL_2 n)^{-1 / 2} \rL_3 n)
\end{equation}
uniformly in $a \in \R$.
In this paper we have tried for a good rate in \refT{T:main}, but we do not know the optimal rate.  The plan for proving~\eqref{BE again 2} will be first to prove~\eqref{BE again 2} uniformly for $|a| > a_n$, where
\begin{equation}
\label{an}
a_n := \frac{1}{2 (d - 1)} (\rL_3 n - 2 \rL_4 n);
\end{equation} 
then to prove
\begin{align}
\lefteqn{\hspace{-.7in}\left| \P(\vp_n > \bzn(a)) - \P(\vp_n > b_n(a)) \right|} \nonumber \\
&= O((\rL_2 n)^{-1 / 2} \rL_3 n)\mbox{\ uniformly for $|a| \leq a_n$}, \label{bzn bn}
\end{align}
where $b_n(a)$ is the following (fairly slight) modification of $\bzn(a)$, employed for computational convenience:
\begin{equation}
\label{bn}
b_n \equiv b_n(a) := \rL n - \rL_3 n - \rL(d - 1) - \rL\!\left( 1 - \frac{a}{\rL_2 n} \right);
\end{equation}
and finally to prove
\begin{align}
\lefteqn{\hspace{-.1in}\left| \P(\vp_n > b_n(a)) - \exp\left[ - \frac{e^{(d - 1) a}}{(d - 1)!} \right] \right|} \nonumber \\
&= O\!\left( (\rL_2 n)^{ - \left( d - \tfrac32 \right)} (\rL_3 n)^{-1} \right) 
= o((\rL_2 n)^{-1 / 2} \rL_3 n)\mbox{\ uniformly for $|a| \leq a_n$}. \label{BE b}
\end{align}

As in \cite[(2.1)]{FNS_2024}, for $b > 0$ define
\begin{equation}
\label{rhonb}
\rho_n(b) \coloneq \#\{ \mbox{maxima~$\rr$ at epoch~$n$ with $\| \rr \| \leq b$} \}.
\end{equation}
Recall that the total variation distance between (the distributions of) two integer-valued random variables $W_1$ and $W_2$ can be defined as the (achieved) supremum
\[
\TV(W_1, W_2) \equiv \TV(\cL(W_1), \cL(W_2)) := \sup_{A \subseteq \Z} |\P(W_1 \in A) - \P(W_2 \in A)| 
\]
(it is sometimes defined as twice this quantity) and that
\begin{align}
\TV(W_1, W_2) 
&= \half \sup_f |\E f(W_1) - \E f(W_2)| = \inf \P(W_1 \neq W_2) \label{TV unequal} \\
&= \half \sum_k | \P(W_1 = k) - \P(W_2 = k) | \label{TV ell1},
\end{align}
where in~\eqref{TV unequal} the (achieved) supremum is taken over all functions $f:\Z \to \R$ with sup-norm bounded by~$1$ and the (achieved) infimum is taken over all couplings of $\cL(W_1)$ and $\cL(W_2)$.
Recall also that (restricting to integer-valued random variables $W, W_1, W_2, \ldots$) convergence of $W_n$ to~$W$ in distribution is equivalent to $\TV(W_n, W) \to 0$. 

We write $\Po(\gl)$ as shorthand for the Poisson distribution with mean~$\gl$ (or, sometimes abusing notation, for a random variable with that distribution).
By a standard ``switching relation'' recasting we have
\begin{equation}
\label{prob}
\P(\vp_n > b_n(a)) = \P(\rho_n(b_n(a)) = 0);
\end{equation}
further, with
\begin{equation}
\label{gl}
\gl \equiv \gl(a) := \frac{e^{(d - 1) a}}{(d - 1)!},
\end{equation}
the subtracted term in~\eqref{BE b} is $e^{- \gl}$, the probability that $\Po(\gl) = 0$.  Thus our key result~\eqref{BE b} will follow immediately from the following Poisson approximation theorem, a theorem we believe is of interest in its own right. 

\begin{theorem}[Poisson approximation]
\label{T:main Poisson}
Fix $d \geq 2$ and recall the notation~\eqref{rhonb}, \eqref{bn}, and~\eqref{gl}.
Then as $n \to \infty$ we have
\begin{equation}
\label{main TV0}
\TV\!\left( \rho_n(b_n), \Po(\gl) \right) 
= O\!\left( (\rL_2 n)^{ - \left( d - \tfrac32 \right)} (\rL_3 n)^{-1} \right)\mbox{\rm \ uniformly for $|a| \leq a_n$}.
\end{equation}
\end{theorem}
\smallskip

\begin{remark}
Using \refT{T:main Poisson}, the rate of convergence in \refT{T:main} improves to match that in \refT{T:main Poisson} if 
$\cL(\vp_n^{\circ})$ is replaced by the distribution of $(\rL_2 n) [1 - \exp(- \vp_n^{\circ} / \rL_2 n)]$.
\end{remark}
\smallskip

For fixed $a \in \R$, it follows from \cite[Lemma~2.1]{FNS_2024} that
\begin{equation}
\label{mean}
\E \rho_n(\bzn) \to \gl
\end{equation}
and from \cite[Lemma~3.1 and Remark~B.1]{FNS_2024} that 
\begin{equation}
\label{variance}
\Var \rho_n(\bzn) \to \gl
\end{equation}
as well.  The theme of~\cite{Arratia_1989}, as stated by the authors, is that ``Convergence to the Poisson distribution, for the number of occurrences of dependent events, can often be established by computing only first and second moments, but not higher ones'', and (notwithstanding the slight discrepancy between $\bzn$ and $b_n$) that theme, together with~\eqref{mean}--\eqref{variance}, is the motivation for the proof of \refT{T:main Poisson}.

Before carrying out the technical details of~\cite{Arratia_1989} to provide rigorous proofs of 
Theorems \ref{T:main}--\ref{T:main Poisson}, which will require a rather substantial amount of modification of $\rho_n(b_n)$ and many additional calculations, let us give a (very) rough sketch as to why $\rho_n(\bzn)$ should be approximately distributed as $\Po(\gl)$, when (as we assume here) $n$ is large.  For $\bun$ defined at 
\eqref{bun}--\eqref{omegan}, $\rho_n(\bun) = 0$ with high probability.  Moreover, for $\x$ satisfying $\bun < \| \x \| \leq \bzn$ and $G_n(\x)$ a small neighborhood of $\x$, the events
\begin{align}
\lefteqn{E_n(\x)} \nonumber \\
&:= \{ \mbox{there is at least one maximum of $\{ \X^{(1)}, \ldots, \X^{(n)} \}$ in $G_n(\x)$} \} \label{Enx}
\end{align}
are ``mostly'' independent as~$\x$ varies (in the sense that there is only suitably local dependence), with 
probabilities
\begin{align}
\lefteqn{\P(E_n(\x))} \nonumber \\ 
&\approx \P(\mbox{there is exactly one maximum of $\{ \X^{(1)}, \ldots, \X^{(n)} \}$ in $G_n(\x)$}) \nonumber \\
&\approx n \P(\X^{(n)} \in G_n(\x),\mbox{\ $\X^{(n)}$ is a maximum of $\{ \X^{(1)}, \ldots, \X^{(n)} \}$}) \nonumber \\
&\approx \int_{\y \in G_n(\x)} n e^{- \| \y \|} (1 - e^{- \| \y \|})^{n - 1} \dd \y \nonumber \\ 
&\approx \int_{\y \in G_n(\x)} (n - 1) e^{- \| \y \|} \exp\!\left[ - (n - 1) e^{- \| \y \|}  \right] \dd \y; \label{PEnx}
\end{align}
thus, $\rho_n(\bzn)$ is approximately Poisson distributed with parameter
\begin{align}
\lefteqn{\hspace{-.5in}\int_{\x:\,\bun < \| \x \| \leq \bzn} (n - 1) e^{- \| \x \|} \exp\!\left[ - (n - 1) e^{- \| \x \|}  \right] \dd \x} 
\nonumber \\
&= \int_{(n - 1) e^{- \bzn}}^{(n - 1) e^{- \bun}} \frac{[\rL(n - 1) - \rL z]^{d - 1}}{(d - 1)!} e^{-z} \dd z \nonumber \\
&\approx \frac{(\rL n)^{d - 1}}{(d - 1)!} \exp\!\left[ - (n - 1) e^{- \bzn} \right]
\approx \frac{e^{(d - 1) a}}{(d - 1)!} = \gl. \label{parameter}
\end{align}

Our rigorous proofs of Theorems~\ref{T:main}--\ref{T:main Poisson} proceed as follows.  In \refS{S:BE1} we 
prove~\eqref{BE again 2} uniformly for $|a| > a_n$.  In \refS{S:BE2} we prove~\eqref{bzn bn}.  The remaining sections, with the exception of the speculative final section, are devoted to the proof of \refT{T:main Poisson}, which will complete the proof of \refT{T:main}.

In \refS{S:1} we introduce a boundary $(\bun)$ that is appreciably smaller than $(b_n)$ and prove that
\begin{equation}
\label{TV1}
\TV\!\left( \rho_n(b_n), \rho_n(b_n) - \rho_n(\bun) \right) = o((\rL n)^{- (d - 1)})\mbox{\ uniformly for $a \in \R$}.
\end{equation}
In \refS{S:2} we introduce a maxima-count $N_n$ that is a Poissonized version of 
$\rho_n(b_n) - \rho_n(\bun)$ and prove that
\begin{equation}
\label{TV2}
\TV\!\left( \rho_n(b_n) - \rho_n(\bun), N_n \right) = o((\rL n)^{- (d - 1)})\mbox{\ uniformly for $|a| \leq a_n$}.
\end{equation}
In \refS{S:3} we introduce a quite large boundary $\big( \bon \big)$ and a random variable $\Non$ that modifies $N_n$ by conditioning on there being no points of the Poisson process with $\ell^1$-norm larger than $\bon$ and prove that
\begin{equation}
\label{TV3}
\TV\!\left( N_n, \Non \right) = O((\rL n)^{-(d - 1)})\mbox{\ uniformly for $|a| \leq a_n$}.
\end{equation}
In \refS{S:4}, by discretization we obtain approximations $Z_{n, m}$ to $\Non$, where each $Z_{n, m}$ (one random variable for each $m \geq 0$)  is a sum of indicator random variables $Z_{n, m, \ga}$, and prove (for each fixed~$n$ and fixed $a \in \R$) that
\begin{equation}
\label{TV4}
\TV\!\left( \Non, Z_{n, m} \right) \to 0\mbox{\ as $m \to \infty$}.
\end{equation}
This discretization and the proof that the dependence of the random variables $Z_{n, m, \ga}$ is suitably local are similar to what was done in \cite{Bai_2005} to prove asymptotic normality of the total number $\rho_n(\infty)$ of maxima, unrestricted by location.   

The heart of the proof of \refT{T:main Poisson} is to employ the technique of \cite{Arratia_1989} and prove that
\begin{align}
\lefteqn{\hspace{-.5in}\limsup_{m \to \infty} \TV\!\left( Z_{n, m}, \Po(\glnm) \right)} \nonumber \\
&= O\!\left( (\rL_2 n)^{ - \left( d - \tfrac32 \right)} (\rL_3 n)^{-1} \right)\mbox{\ uniformly for $|a| \leq a_n$} \label{TV5}
\end{align}
as $n \to \infty$, where $\glnm := \E Z_{n, m}$.  Our proof of~\eqref{TV5} will use results from other sections, including Sections~\ref{S:6}--\ref{S:7}, so we defer it to \refS{S:5}.

In \refS{S:6} we prove (for each fixed~$n$ and fixed $a \in \R$) that
\begin{equation}
\label{TV6}
\TV\!\left( \Po(\glnm), \Po(\E \Non) \right) \to 0\mbox{\ as $m \to \infty$}.
\end{equation}
In \refS{S:7} we prove that
\begin{equation}
\label{TV7}
\TV\!\left( \Po(\E \Non), \Po(\gl) \right) = O((\rL n)^{-1} (\rL_2 n)^{1/2})\mbox{\ uniformly for $|a| \leq a_n$}.
\end{equation}

\refT{T:main Poisson} then follows by first applying the triangle inequality to $\TV\!\left( \rho_n(b_n), \Po(\gl) \right)$ using \eqref{TV1}--\eqref{TV7} (for fixed~$n$ and fixed $|a| \leq a_n$), then taking limit superior as $m \to \infty$ on both sides of the inequality, and finally taking the supremum of both sides over $|a| \leq a_n$ on both sides of the inequality.

Our proofs use several results that are either taken from \cite{FNS_2024} or closely related to results from \cite{FNS_2024}.  For the reader's convenience, such results are collected in \refApp{A:FNS}.

In \refS{S:location} we consider the distribution of the 
maximum at epoch~$n$ having smallest 
$\ell^1$-norm, not just (as in all other sections) its $\ell^1$-norm.

\section{Distribution tails}
\label{S:BE1}

\begin{proposition}
\label{P:BE again 2}
With $a_n$ defined at~\eqref{an}, we have
\begin{equation}
\label{BE again 2OLD}
\left| \P(\vp_n > \bzn(a)) - e^{- \gl(a)} \right|
= O((\rL_2 n)^{-1 / 2} \rL_3 n)
\end{equation}
uniformly for $|a| > a_n$.
\end{proposition}

\begin{proof}
First suppose that $a > a_n$.  Then, using \refL{L:moment method},
\begin{align*}
\P(\vp_n > \bzn(a)) 
&\leq \P(\vp_n > \bzn(a_n)) \\
&\leq (1 + o(1))\,(d - 1)! (\rL n)^{ - (d - 1) \left( 1 - e^{ - \frac{a_n}{\rL_2 n}} \right)} \\
&\sim (d - 1)! \exp[- (d - 1) a_n ] \\
&= (d - 1)! (\rL_2 n)^{-1/2} \rL_3 n;
\end{align*}
and
\begin{align*}
e^{- \gl(a)}
&\leq e^{- \gl(a_n)} 
= \exp\!\left[ - \frac{(\rL_2 n)^{1/2} (\rL_3 n)^{-1}}{(d - 1)!} \right] = o\!\left( (\rL_2 n)^{-1/2} \rL_3 n \right).
\end{align*}

Now suppose instead that $a < - a_n$.  Then, again using \refL{L:moment method},
\begin{align*}
1 - \P(\vp_n > \bzn(a)) 
&= \P(\vp_n \leq \bzn(a)) \\ 
&\leq \P(\vp_n \leq \bzn(- a_n)) \\
&\leq (1 + o(1))\,\frac{1}{(d - 1)!} (\rL n)^{ - (d - 1) \left( e^{\frac{a_n}{\rL_2 n}} - 1 \right)} \\
&\sim (d - 1)! \exp[- (d - 1) a_n ] \\
&= (d - 1)! (\rL_2 n)^{-1/2} \rL_3 n;
\end{align*}
and
\begin{align*}
1 - e^{- \gl(a)}
&\leq 1 - e^{- \gl(-a_n)} \\
&= 1 - \exp\!\left[ - \frac{(\rL_2 n)^{-1/2} \rL_3 n}{(d - 1)!} \right] \\
&\sim \frac{(\rL_2 n)^{-1/2} \rL_3 n}{(d - 1)!}. 
\end{align*}
The proof of the proposition is complete.
\end{proof}

\section{Slight modification $(b_n)$ of the boundary $(\bzn)$}
\label{S:BE2}
For 
our 
Poisson approximation theorem, \refT{T:main Poisson}, we find it a bit more convenient to use the boundary
\begin{equation}
\label{bnOLD}
b_n(a) := \rL n - \rL_3 n - \rL(d - 1) - \rL\!\left( 1 - \frac{a}{\rL_2 n} \right)
\end{equation}
rather than~\eqref{b0n}.  
In this section we establish~\eqref{bzn bn} as the following proposition.

\begin{proposition}
\label{P:bzn bn}
We have
\begin{align}
\lefteqn{\hspace{-.7in}\left| \P(\vp_n > \bzn(a)) - \P(\vp_n > b_n(a)) \right|} \nonumber \\
&= O((\rL_2 n)^{-1 / 2} \rL_3 n)\mbox{\rm \ uniformly for $|a| \leq a_n$}. \label{bzn bnOLD}
\end{align}
\end{proposition}

\begin{proof}
First note that for every~$a$ satisfying $|a| \leq a_n$ we have $b_n(a) \geq \bzn(a)$, and so the left side of~\eqref{bzn bnOLD} equals
\begin{align*}
\P(\bzn(a) < \vp_n \leq b_n(a)) 
&= \P(\rho_n(b_n(a)) \geq 1,\,\rho_n(\bzn(a)) = 0) \\
&\leq \P(\rho_n(b_n(a)) - \rho_n(\bzn(a)) \geq 1) \\
&\leq \E\!\left[ \rho_n(b_n(a)) - \rho_n(\bzn(a)) \right]
=: \gD_n(a) \equiv \gD_n.
\end{align*}
Now
\begin{align*}
\gD_{n + 1} 
&= \frac{(n + 1)}{(d - 1)!} \int_{b^*_{n + 1}}^{b_{n + 1}} y^{d - 1} e^{-y} (1 - e^{-y})^n \dd y \\
&\leq \frac{(n + 1)}{(d - 1)!} \int_{b^*_{n + 1}}^{b_{n + 1}} y^{d - 1} \exp(- n e^{-y} - y) \dd y \\
&= \frac{1 + n^{-1}}{(d - 1)!} \int_{n e^{- b_{n + 1}}}^{n e^{- b^*_{n + 1}}} (\rL n - \rL z)^{d - 1} e^{-z} \dd z \\
&\leq (1 + n^{-1}) \frac{(\rL n)^{d - 1}}{(d - 1)!} \int_{n e^{- b_{n + 1}}}^{n e^{- b^*_{n + 1}}} e^{-z} \dd z \\
&\leq (1 + n^{-1}) \frac{(\rL n)^{d - 1}}{(d - 1)!} \exp(- n e^{- b_{n + 1}}) \times n (e^{- b^*_{n + 1}} - e^{- b_{n + 1}}). 
\end{align*}
But
\begin{align*}
e^{- b_{n + 1}} 
&= (n + 1)^{-1} [\rL_2(n + 1)] (d - 1) \left( 1 - \frac{a}{\rL_2(n + 1)} \right) \\
&\geq (n + 1)^{-1} [\rL_2(n + 1)] (d - 1) \left( 1 - \frac{a_{n + 1}}{\rL_2(n + 1)} \right) \\
\end{align*}
and
\begin{align*}
\lefteqn{e^{- b^*_{n + 1}} - e^{- b_{n + 1}}} \\
&\leq e^{- b^*_{n + 1}} (b_{n + 1} - b^*_{n + 1}) \\
&= (n + 1)^{-1} [\rL_2(n + 1)] (d - 1) e^{- a / \rL_2 n} \left[- \rL\!\left( 1 - \frac{a}{\rL_2(n + 1)} \right) - \frac{a}{\rL_2(n + 1)} \right] \\
&\leq (1 + o(1)) (n + 1)^{-1} [\rL_2(n + 1)] (d - 1) 
\left[- \rL\!\left( 1 - \frac{a_{n + 1}}{\rL_2(n + 1)} \right) - \frac{a_{n + 1}}{\rL_2(n + 1)} \right] \\
&\sim \frac12 (d - 1) (n + 1)^{-1} [\rL_2(n + 1)]^{-1} a_{n + 1}^2 \\
&\sim \frac{1}{8 (d - 1)} (n + 1)^{-1} [\rL_2(n + 1)]^{-1} [\rL_3(n + 1)]^2,  
\end{align*}
so
\begin{align*}
\gD_{n + 1}
&\leq (1 + o(1)) \frac{[\rL_2(n + 1)]^{-1/2} \rL_3(n + 1)}{8 (d - 1) (d - 1)!}
\end{align*}
and
\[
\gD_n = O((\rL_2 n)^{-1/2} \rL_3 n),
\]
as desired.
\end{proof}

\section{A boundary $(\bun)$ that is appreciably smaller than $(b_n)$}
\label{S:1}

In this section we introduce a boundary $(\bun)$ (not depending on~$a$) that is appreciably smaller than $(b_n)$ (when $a \geq - a_n$) and prove that
\begin{equation}
\label{TV11}
\TV\!\left( \rho_n(b_n), \rho_n(b_n) - \rho_n(\bun) \right) = o((\rL n)^{- (d - 1)})\mbox{\ uniformly in $a \in \R$},
\end{equation}
which is~\eqref{TV1}.  Indeed, let
\begin{equation}
\label{bun}
\bun := \rL n - \rL_3 n - \rL \go_n,
\end{equation}
where the sequence $(\go_n)$ is chosen (not depending on~$a$) arbitrarily subject to the growth conditions
\begin{equation}
\label{omegan}
\go_n \to \infty\mbox{\ \ and\ \ }\rL \go_n = o(\rL_3 n).
\end{equation}
Then
\[
\TV\!\left( \rho_n(b_n), \rho_n(b_n) - \rho_n(\bun) \right) \leq \P(\rho_n(\bun) \neq 0) \leq \E \rho_n(\bun).
\]
By \refL{L:Erhonbun},
\[
\E \rho_n(\bun) = O((\rL n)^{d - 1 - \go_n}) = o((\rL n)^{- (d - 1)}),
\]
completing the proof of~\eqref{TV11}.

\begin{remark}
\label{R:Erhonbun}
Though we will not need it (nor give a proof), more accurate asymptotic determination of $\E \rho_n(\bun)$ is possible, namely,
\[
\E \rho_n(\bun) = \frac{1}{(d - 1)!} (\rL n)^{d - 1 - \go_n} + O((\rL n)^{d - 2 - \go_n} \rL_3 n). 
\] 
\end{remark} 

\section{Poissonization}
\label{S:2}

In this section we introduce a maxima-count $N_n$ that is a Poissonized version of 
$\rho_n(b_n) - \rho_n(\bun)$ and prove that
\begin{equation}
\label{TV22}
\TV\!\left( \rho_n(b_n) - \rho_n(\bun), N_n \right) = o((\rL n)^{- (d - 1)})\mbox{\ uniformly for $|a| \leq a_n$},
\end{equation}
which is~\eqref{TV2}.

Indeed, let $N_n$ count the maxima from a Poisson process with intensity measure
\begin{equation}
\label{mun}
\mu_n(\ddx \x) := n e^{- \| \x \|} {\bf 1}(\x \succ {\bf 0}) \dd \x,
\end{equation}
that belong to the set
\begin{equation}
\label{An}
A_n \equiv A_n(a) := \{ \x \succ {\bf 0}: \bun < \| \x \| \leq b_n \},
\end{equation}
a set which is nonempty because $a \geq - a_n$.

\begin{proof}[Proof of~\eqref{TV22}]
Let
\begin{equation}
\label{Bn}
B_n := \{ \x \succ {\bf 0}: \| \x \| > \bun \},
\end{equation}
and observe that (with no dependence on $a \in \R$)
\begin{align}
p_n 
&:= n^{-1} \mu_n(B_n) \nonumber \\
&= \P(\Gam(d, 1) > \bun) = \frac{1}{(d - 1)!} \int_{\bun}^{\infty} x^{d - 1} e^{-x} \dd x \nonumber \\
&\sim \frac{\bun^{d - 1}}{(d - 1)!} e^{- \bun} 
\sim \frac{1}{(d - 1)!} (\rL n)^{d - 1} n^{-1} (\rL_2 n) \go_n \nonumber \\
&= o((\rL n)^{d - 1 - \go_n}) = o((\rL n)^{- (d - 1)}). \label{pn bound}
\end{align}

To establish~\eqref{TV22}, we first note that observations (whether $\X^{(i)}$'s or Poisson points) with $\ell^1$-norm bounded above by $\bun$ play no role in determining which observations falling in $A_n$ are maxima.  So we can 
proceed as in the proof of \cite[Theorem~4.1]{Sun_2025} and
note (see, e.g.,\ \cite[Proposition~3.6]{Resnick_2008} for justification) that, for each nonnegative integer~$k$, the conditional distribution of $N_n$ given that~$k$ points of the Poisson process fall in $B_n$
is exactly the conditional distribution of $\rho_n(b_n) - \rho_n(\bun)$
given that~$k$ of the~$n$ observations $\X^{(1)}, \ldots, \X^{(n)}$ fall in $B_n$.  Thus, by conditioning 
and~\eqref{TV unequal}, $\TV\!\left( \rho_n(b_n) - \rho_n(\bun), N_n \right)$ is bounded by twice the total variation distance between the counts of values falling in $B_n$; these two counts have distributions Binomial$(n, p_n)$ and Poisson$(n p_n)$.  Therefore, because the total variation distance between these Binomial and Poisson distributions is bounded by $p_n$ (see, e.g.,\ 
\cite[Chapter~1, (1.23)]{Barbour_1992}), we conclude from~\eqref{pn bound} that
\[
\TV\!\left( \rho_n(b_n) - \rho_n(\bun), N_n \right) \leq 2 p_n = o((\rL n)^{- (d - 1)}),
\]
with the bound not depending on~$a$, as claimed at~\eqref{TV22}.
\end{proof}

We have observed that we may, and henceforth do, take $N_n$ to count the maxima from a Poisson process with intensity measure
\begin{equation}
\label{muun}
\muun(\ddx \x) := n e^{- \| \x \|} {\bf 1}(\x \succ {\bf 0},\,\| \x \| > \bun) \dd \x
\end{equation}
that belong to the set $A_n$ of~\eqref{An}.

\section{Conditioning}
\label{S:3}

In this section we introduce 
\begin{equation}
\label{bon}
\bon := \rL n + 2 (d - 1) \rL_2 n
\end{equation}
and $\Non$ such that, for each~$n$, the law of the random variable $\Non$ is the conditional distribution of $N_n$ given that the Poisson point process with intensity measure~\eqref{muun} has no points with $\ell^1$-norm larger than $\bon$.  Put another way, we can (and do) define $\Non$ to be the number of maxima of the Poisson point process with intensity measure
\begin{equation}
\label{muon}
\muon(\ddx \x) := n e^{- \| \x \|} {\bf 1}(\x \succ {\bf 0},\,\bun < \| \x \| \leq \bon) \dd \x
\end{equation}
that belong to the set $A_n$ of~\eqref{An}.  The following proposition establishes~\eqref{TV3}.

\begin{proposition}
\label{P:TV33}
We have
\begin{equation}
\label{TV33}
\TV\!\left( N_n, \Non \right) = O((\rL n)^{-(d - 1)})\mbox{\rm \ uniformly for $|a| \leq a_n$}.
\end{equation}
\end{proposition}

\begin{proof}
It is easy to check using~\eqref{TV ell1} that $\TV\!\left( N_n, \Non \right)$ is bounded by 
\begin{equation}
\label{Eratio}
\frac{\P(E_n)}{1 - \P(E_n)},
\end{equation}
where $E_n$ is the event that the Poisson process with intensity measure~\eqref{mun} has a maximum with $\ell^1$-norm larger than $\bon$; this bound does not depend on~$a$.  A hint that~\eqref{Eratio} is small when~$n$ is large is provided by 
\cite[Theorem~1.8]{Fillboundary_2020} [recall~\eqref{F+n}], which, however, deals with the non-Poissonized model.  Indeed, such a result follows from \refL{L:PEn}, implying a bound of $O((\rL n)^{- (d - 1)})$ on~\eqref{Eratio} and thereby completing the proof.
\end{proof}

\section{Discretization}
\label{S:4}

In this section, by discretization we obtain approximations $Z_{n, m}$ to $\Non$, where each $Z_{n, m}$ (one random variable for each $m \geq 0$)  is a sum of indicator random variables $Z_{n, m, \ga}$, and prove (for fixed~$n$ and~$a$) that~\eqref{TV4} holds:
\begin{equation}
\label{TV44}
\TV\!\left( \Non, Z_{n, m} \right) \to 0\mbox{\ as $m \to \infty$}.
\end{equation}

Fix~$n$ and define
\begin{equation}
\label{Aon}
\Aon := \{ \x \succ {\bf 0}:\bun < \| \x \| \leq \bon \}.
\end{equation}
Enclose $\Aon$ in a cube of fixed finite side-length.  For each integer $m \geq 0$, decompose this cube into sub-cubes $G_{n, m, \ga}$ of side-length $2^{-m}$.  (To what extent each $G_{n, m, \ga}$ contains its boundary will be immaterial.)  Let $\Pion$ denote a Poisson process with intensity measure~\eqref{muon}, and let  $\Mon$ denote the random measure that counts maxima of $\Pion$---so that, in particular, 
\begin{equation}
\label{NonMon}
\Non = \Mon(A_n).  
\end{equation}
With~$\wedge$ denoting 
minimum, let
\begin{equation}
\label{Znmga}
Z_{n, m, \ga} := 1 \wedge \Mon(A_n \cap G_{n, m, \ga})
\end{equation}
and
\begin{equation}
\label{Inm}
I_{n, m} := \{ \ga:A_n \cap G_{n, m, \ga} \neq \varnothing \}.
\end{equation} 
All operations over index~$\ga$ henceforth are taken over the index set $I_{n, m}$ unless otherwise specified.
Let
\begin{equation}
\label{Znm}
Z_{n, m} := \sum_{\ga} Z_{n, m, \ga}.
\end{equation}

\begin{proof}[Proof of~\eqref{TV44}]
Using~\eqref{TV unequal} for the first inequality, we have
\begin{align}
\lefteqn{\hspace{-.5in}1 - \TV\!\left( \Non, Z_{n, m} \right)} \nonumber \\
&\geq \P(Z_{n, m} = \Non) \label{P1} \\
&\geq \P\left( \cap_{\ga} \left\{ \Pion \left( \Aon \cap G_{n, m, \ga} \right) \leq 1 \right\} \right) \nonumber \\
&= \exp\left[ - \sum_{\ga} \muon\left( \Aon \cap G_{n, m, \ga} \right) \right] 
\prod_{\ga} \left[ 1 + \muon\left( \Aon \cap G_{n, m, \ga} \right) \right] \nonumber \\ 
&\to 1\mbox{\ as $m \to \infty$}, \label{P2}
\end{align}
as desired, where the last-step convergence follows by routine application of 
\cite[Lemma preceding Theorem 7.1.2]{Chung_2001}.
\end{proof}

\section{A sequence of Poisson distributions indexed by~$m$}
\label{S:6}

In this section, recalling the notation $\gl_{n, m} = \E Z_{n, m}$, we prove (for each fixed~$n$ and~$a$) that~\eqref{TV6} holds:
\begin{equation}
\label{TV66}
\TV\!\left( \Po(\glnm), \Po(\E \Non) \right) \to 0\mbox{\ as $m \to \infty$}.
\end{equation}

\begin{proof}
From the argument at \eqref{P1}--\eqref{P2} we see convergence in probability:
\[
Z_{n, m} \Pto \Non\mbox{\ as $m \to \infty$}.
\]
Noting from~\eqref{Znmga} that $Z_{n, m}$ is monotonically nondecreasing in~$m$, the same convergence holds almost surely (see \cite[Problem~20.22(a)]{BillingsleyPM_2012}).  Therefore, by the monotone convergence theorem, we also have $L^1$-convergence and hence
\begin{equation}
\label{convergence of means}
\gl_{n, m} \uparrow \E \Non\mbox{\ as $m \uparrow \infty$}.
\end{equation}
Thus, utilizing the semigroup property of Poisson distributions under convolution and~\eqref{TV unequal} at the initial step,
\begin{align}
\TV\!\left( \Po(\glnm), \Po(\E \Non) \right)
&\leq \P\left( \Po\!\left( \E \Non - \gl_{n, m} \right) \neq 0 \right) \label{P3} \\
&= 1 - \exp\!\left[- \left( \E \Non - \gl_{n, m} \right) \right] \label{P4} \\
&\to 0\mbox{\ as $m \to \infty$}. \nonumber
\pushQED{\qed} 
\qedhere
\popQED
\end{align}
\noqed
\end{proof}

\section{A sequence of Poisson distributions indexed by~$n$}
\label{S:7}

In this section we prove that~\eqref{TV7} holds:
\begin{equation}
\label{TV77}
\TV\!\left( \Po(\E \Non), \Po(\gl) \right) = O((\rL n)^{-1} (\rL_2 n)^{1/2})\mbox{\ uniformly for $|a| \leq a_n$}.
\end{equation}
By the same sort of argument as at \eqref{P3}--\eqref{P4}, we need only prove the following lemma.

\begin{lemma}
\label{L:two means}
As $n \to \infty$, we have
\begin{equation}
\label{two means}
\E \Non = \gl + O((\rL n)^{-1} (\rL_2 n)^{1/2})\mbox{\rm \ uniformly for $|a| \leq a_n$}.
\end{equation}
\end{lemma}

\begin{proof}
Given $\x \succ {\bf 0}$, let
\[
O^+_{\x} := \{ \y:\y \succ \x \}
\]
denote the positive orthant in $\R^d$, translated by~$\x$.
Recalling display~\eqref{NonMon} and the notation $\Pion$ introduced just before it,
\begin{equation}
\label{NonMon calculation}
\E \Non
= \E \Mon(A_n)
= \int_{A_n} \P\!\left( \Pion(O^+_{\x}) = 0 \right)\,\muon(\ddx \x).
\end{equation}
Now for $\x \in A_n$ we have
\begin{equation}
\label{PPionO+x=0}
\P\!\left( \Pion(O^+_{\x}) = 0 \right) = \exp\!\left[ - \int_{\y \in \Aon:\y \succ \x}\,\muon(\ddx \y) \right],
\end{equation}
and the integral appearing in~\eqref{PPionO+x=0} equals
\begin{equation}
\label{integral1}
n e^{- \| \x \|} \int_{\z \succ 0:\,\| \z \| \leq \bon - \| \x \|} e^{- \| \z \|} \dd \z.  
\end{equation}
The integral factor in~\eqref{integral1} equals $\P(\Gam(d, 1) \leq \bon - \| \x \|)$ which, for every $\x \in A_n$, lies between the upper bound of~$1$ and the lower bound of
\begin{align}
\lefteqn{\hspace{-.5in}\P(\Gam(d, 1) \leq \bon - b_n)} \nonumber \\
&= 1 - \P(\Gam(d, 1) > 2 (d - 1) \rL_2 n + \rL_3 n + \rL c_n); \label{LB1}
\end{align}
further, uniformly for $|a| \leq a_n$ we have
\begin{align}
\lefteqn{\P(\Gam(d, 1) > 2 (d - 1) \rL_2 n + \rL_3 n + \rL c_n)} \nonumber \\
&\sim \frac{1}{(d - 1)!} [2 (d - 1) \rL_2 n]^{d - 1} (\rL n)^{- 2 (d - 1)} (\rL_2 n)^{-1} c_n^{-1} \nonumber \\
&= \Theta\!\left( (\rL n)^{- 2 (d - 1)} (\rL_2 n)^{d - 2} \right). \label{Theta}
\end{align}
Putting the pieces of~\eqref{NonMon calculation}--\eqref{Theta} together and recalling~\eqref{muon}, we find
\begin{equation}
\label{PPionO+x=0 2}
\P\!\left( \Pion(O^+_{\x}) = 0 \right)
= \exp\!\left[ - n e^{- \| \x \|} (1 - u_{n, a}(\| \x \|)) \right]\mbox{\ for $\x \in A_n$}
\end{equation}
with
\begin{equation}
\label{uny}
u_{n, a}(y) = O\!\left( (\rL n)^{- 2 (d - 1)} (\rL_2 n)^{d - 2} \right)\mbox{\ uniformly for $y \in (\bun, b_n]$ and $|a| \leq a_n$}
\end{equation}
and thence
\begin{equation}
\label{ENon}
\E \Non 
= n \int_{A_n} e^{- \| \x \|} \exp\!\left[ - n e^{- \| \x \|} (1 - u_{n, a}(\| \x \|)) \right] \dd \x.
\end{equation}

Continuing from~\eqref{ENon},
\begin{equation}
\label{ENon2}
\E \Non 
= n \int_{\bun}^{b_n} \frac{y^{d - 1}}{(d - 1)!} e^{- y} \exp\!\left[ - (1 - u_{n, a}(y)) n e^{- y} \right] \dd y.
\end{equation}
Let
\begin{equation}
\label{han}
\han := \sup_{|a| \leq a_n,\,y \in (\bun, b_n]} |u_{n, a}(y)| = O\!\left( (\rL n)^{- 2 (d - 1)} (\rL_2 n)^{d - 2} \right). 
\end{equation}
Then, applying the change of variables $y = \rL n - \rL z + \rL(1 + \han)$, we find, uniformly for $|a| \leq a_n$, that
\begin{align*}
\lefteqn{\E \Non} \\ 
&\geq \frac{(1 + \han)^{-1}}{(d - 1)!} \int_{(1 + \han) \gb_n}^{(1 + \han) \gbun} [\rL n - \rL z + \rL(1 + \han)]^{d - 1} e^{- z} \dd z \\
&\geq \frac{(1 + \han)^{-1}}{(d - 1)!} (\rL n - \rL \gbun)^{d - 1} 
\left\{ \exp\!\left[- (1 + \han) \gb_n \right] - \exp\!\left[- (1 + \han) \gbun \right] \right\} \\
&\geq \frac{(1 + \han)^{-1}}{(d - 1)!} (\rL n - \rL \gbun)^{d - 1} 
\left\{ \exp\!\left[- (1 + \han) \gb_n \right] - e^{- \gbun} \right\} \\
&= \frac{\left[ 1 - O\!\left( (\rL n)^{- 2 (d - 1)} (\rL_2 n)^{d - 2} \right) \right]}{(d - 1)!}  
(\rL n)^{d - 1} \left[1 - (1 - o(1)) (d - 1) \frac{\rL_3 n}{\rL n} \right] \\
&{} \qquad \times 
e^{ - \gb_n} \left\{ \left[ 1 - O\!\left( (\rL n)^{- 2 (d - 1)} (\rL_2 n)^{d - 1} \right) \right] - e^{- (\gbun - \gb_n)} \right\} \\
&= \left[1 - (1 - o(1)) (d - 1) \frac{\rL_3 n}{\rL n} \right] \frac{(\rL n)^{d - 1 - c_n}}{(d - 1)!} \\
&= \left[1 - (1 - o(1)) (d - 1) \frac{\rL_3 n}{\rL n} \right] \gl 
= \gl -  O\!\left( \gl \frac{\rL_3 n}{\rL n} \right).
\end{align*}

Similarly, applying the change of variables $y = \rL n - \rL z + \rL(1 - \han)$, we find, uniformly for $|a| \leq a_n$, that
\begin{align*}
\lefteqn{\E \Non} \\ 
&\leq \frac{(1 - \han)^{-1}}{(d - 1)!} \int_{(1 - \han) \gb_n}^{(1 - \han) \gbun} [\rL n - \rL z + \rL(1 - \han)]^{d - 1} e^{- z} \dd z \\
&\leq \frac{(1 - \han)^{-1}}{(d - 1)!} (\rL n)^{d - 1} 
\left\{ \exp\!\left[- (1 - \han) \gb_n \right] - \exp\!\left[- (1 - \han) \gbun \right] \right\} \\
&\leq \frac{(1 - \han)^{-1}}{(d - 1)!} (\rL n)^{d - 1} 
\left\{ \exp\!\left[- (1 - \han) \gb_n \right] - e^{- \gbun} \right\} \\
&= \frac{\left[ 1 + O\!\left( (\rL n)^{- 2 (d - 1)} (\rL_2 n)^{d - 2} \right) \right]}{(d - 1)!}  
(\rL n)^{d - 1} \\
&{} \qquad \times 
e^{ - \gb_n} \left\{ \left[ 1 + O\!\left( (\rL n)^{- 2 (d - 1)} (\rL_2 n)^{d - 1} \right) \right] - e^{- (\gbun - \gb_n)} \right\} \\
&= \left[ 1 + O\!\left( (\rL n)^{- 2 (d - 1)} (\rL_2 n)^{d - 1} \right) \right] \frac{(\rL n)^{d - 1 - c_n}}{(d - 1)!} \\
&= \left[ 1 + O\!\left( (\rL n)^{- 2 (d - 1)} (\rL_2 n)^{d - 1} \right) \right] \gl \\
&= \gl +  O\!\left( \gl (\rL n)^{- 2 (d - 1)} (\rL_2 n)^{d - 1} \right) = \gl + o\!\left( \gl \frac{\rL_3 n}{\rL n} \right).
\end{align*}

But for $a \leq a_n$ we have
\[
\gl \leq \frac{e^{(d - 1) a_n}}{(d - 1)!} = \frac{(\rL_2 n)^{1/2} (\rL_3 n)^{-1}}{(d - 1)!},
\]
so we conclude that
\[
\E \Non = \gl + O((\rL n)^{-1} (\rL_2 n)^{1/2}),
\]
completing the proof of~\eqref{two means}.
\end{proof}

\section{Poisson approximation for $Z_{n, m}$}
\label{S:5}

In this section, centrally important to the proof of \refT{T:main Poisson}, we employ the technique of \cite{Arratia_1989} to prove that~\eqref{TV5} holds, namely, that
\begin{align}
\lefteqn{\hspace{-.5in}\limsup_{m \to \infty} \TV\!\left( Z_{n, m}, \Po(\glnm) \right)} \nonumber \\ 
&= O\!\left( (\rL_2 n)^{ - \left( d - \tfrac32 \right)} (\rL_3 n)^{-1} \right)\mbox{\ uniformly for $|a| \leq a_n$} \label{TV55}
\end{align}
as $n \to \infty$, where $\glnm := \E Z_{n, m}$.  We will make a reasonable effort to match the notation used in~\cite{Arratia_1989}, but, since we have already used $b_n$ for the boundary~\eqref{bn}, we will use fraktur font to write 
$\fb_1, \fb_2, \fb_3$ for the three key quantities $b_1, b_2, b_3$ in~\cite{Arratia_1989}. 

The goal, then, of this section is to establish a Poisson approximation, with error bound, for $Z_{n, m}$ defined at~\eqref{Znm}, with $Z_{n, m, \ga}$ defined at~\eqref{Znmga}.  Treating the sub-cube $G_{n, m, \ga}$ as a closed cube, let $\gu_{n, m, \ga}$ denote its unique minimum (\ie,\ ``southwestern-most point'').  Note first that if we make, for each $\ga \in I_{n, m}$ (with $I_{n, m}$ defined at~\eqref{Inm}), the definition
\begin{equation}
\label{Bnmga}
B_{n, m, \ga} := \{ \gam \in I_{n, m}:\| \gu_{n, m, \ga} \vee \gu_{n, m, \gam} \| < \bon \}, 
\end{equation}
where $\vee$ denotes coordinate-wise maximum, then
\[
\mbox{for every $\ga \in I_{n, m}$:\ $Z_{n, m, \ga}$ is independent of 
$\gs\langle Z_{n, m, \gam}:\gam \in I_{n, m} - B_{n, m, \ga} \rangle$},
\] 
so $\fb_3 = 0$ in the notation of~\cite{Arratia_1989}.  In~\cite{Arratia_1989}, see Theorem~1 and the penultimate full paragraph on p.~12 of~\cite{Arratia_1989}, from which we also find (noting the factor-of-$2$ difference in their definition of total variation distance) that
\begin{equation}
\label{TVfb}
\TV\!\left( Z_{n, m}, \Po(\glnm) \right) \leq \fb_1 + \fb_2 = \E Z_{n, m}^2 - \E \Po(\glnm)^2 + 2 \fb_1.
\end{equation}
With
\begin{equation}
\label{pnmga}
p_{n, m, \ga} := \P(Z_{n, m, \ga} = 1)
\end{equation}
and
\begin{equation}
p_{n, m, \ga, \gam} := \E(Z_{n, m, \ga} Z_{n, m, \gam}),
\end{equation}
the notation in~\eqref{TVfb} is that
\begin{align}
\fb_1 \equiv \fb_1(n, m) 
&= \sum_{\ga \in I_{n, m}} \sum_{\gam \in B_{n, m, \ga}} p_{n, m, \ga}\,p_{n, m, \gam}, \label{fb1def} \\
\fb_2 \equiv \fb_2(n, m) 
&= \sum_{\ga \in I_{n, m}} \sum_{\gam \neq \ga:\,\ga \in B_{n, m, \ga}} p_{n, m, \ga, \gam}. \nonumber
\end{align}
Note that if we use the second of the two expressions for the upper bound in~\eqref{TVfb}, then we need treat only $\fb_1$, not $\fb_2$.

For the first of the three terms in that bound, we claim that
\begin{equation}
\label{1 of 3}
\E Z_{n, m}^2 \to \E \Non^2\mbox{\ as $m \to \infty$}.
\end{equation}
Indeed, by the same argument leading to~\eqref{convergence of means}, $Z_{n, m}$ converges in $L^2$ to $\Non$, and hence (by the second assertion in~\cite[Theorem 4.5.1]{Chung_2001} with $r = 2$) \eqref{1 of 3} holds.

The (subtracted) second of the three terms is
\begin{equation}
\label{2 of 3}
\E \Po(\glnm)^2 = \glnm + \glnm^2 \to \E \Non + \left( \E \Non \right)^2\mbox{\ as $m \to \infty$},
\end{equation}
where the convergence here follows from~\eqref{convergence of means}.  Therefore, by~\eqref{1 of 3}--\eqref{2 of 3},
\begin{equation}
\label{1--2 of 3}
\E Z_{n, m}^2 - \E \Po(\glnm)^2 \to \Var \Non - \E \Non\mbox{\ as $m \to \infty$}. 
\end{equation}

For the proof of~\eqref{TV55}, we need also \refL{L:VarNon}, which establishes
\begin{equation}
\label{VarNon again}
\Var \Non = \E \Non + O\!\left( (\rL_2 n)^{ - \left( d - \tfrac32 \right)} (\rL_3 n)^{-1} \right)\mbox{\ uniformly for $|a| \leq a_n$},
\end{equation}
and the following lemma, whose proof is given in \refS{S:5.1}.

\begin{lemma}
\label{L:limsup fb1}
As $n \to \infty$ we have
\begin{equation}
\label{limsup fb1}
\limsup_{m \to \infty} \fb_1(n, m) = O((\rL n)^{- (d - 1)} (\rL_2 n)^d (\rL_3 n)^{-2})\mbox{\rm \ uniformly for $|a| \leq a_n$}.
\end{equation}
\end{lemma}

We are now ready for the proof of~\eqref{TV55}.
\begin{proof}[Proof of~\eqref{TV55}]
Simply combine~\eqref{TVfb} and~\eqref{1--2 of 3}--\eqref{limsup fb1}.
\end{proof}

\subsection{The quantity $\fb_1(n, m)$}
\label{S:5.1}

\begin{proof}[Proof of \refL{L:limsup fb1}]
Recall from~\eqref{fb1def}, \eqref{Bnmga}, \eqref{pnmga}, and \eqref{Znmga} that
\begin{align*}
\lefteqn{\fb_1(n, m)} \\ 
&= \sum_{\ga \in I_{n, m}}\,\sum_{\substack{\gam \in I_{n, m}: \\ \| \gu_{n, m, \ga} \vee \gu_{n, m, \gam} \| < \bon}} 
\P(\Non(A_n \cap G_{n, m, \ga}) \geq 1)\,\P(\Non(A_n \cap G_{n, m, \gam}) \geq 1).
\end{align*}
Again treating the sub-cube $G_{n, m, \ga}$ as a closed cube, let $\gov_{n, m, \ga}$ denote its unique maximum (\ie,\ ``northeastern-most point'').  Then for $\ga \in I_{n, m}$ we have
\[
\{ \Non(A_n \cap G_{n, m, \ga}) \geq 1 \} 
\subseteq \{ \Pion(A_n \cap G_{n, m, \ga}) \geq 1 \} \cap \left\{ \Pion\!\left( O^+_{\gov_{n, m, \ga}} \right) = 0 \right\}
\] 
and so
\begin{align}
\lefteqn{\P(\Non(A_n \cap G_{n, m, \ga}) \geq 1)} \nonumber \\
&\leq \left[ 1 - \exp\!\left(- \int_{\x \in G_{n, m, \ga}} n e^{- \| \x \|} \dd \x \right) \right]
\exp\!\left(- \int_{\y \succ \gov_{n, m, \ga}:\,\| \y \| \leq \bon} n e^{- \| \y \|} \dd \y \right). \label{P6}
\end{align}

Consider the two factors on the right in~\eqref{P6}.  Since the volume of the cube $G_{n, m, \ga}$ is $2^{- d m}$, the first factor is bounded by
\[
\int_{\x \in G_{n, m, \ga}} n e^{- \| \x \|} \dd \x \leq n 2^{- d m} \exp\!\left( - \left\| \gu_{n, m, \ga} \right\| \right).
\]
The integral appearing in the second factor equals
\[
n \exp\!\left( - \left\| \gov_{n, m, \ga} \right\| \right) \P\!\left( \Gam(d, 1) \leq \bon - \left\| \gov_{n, m, \ga} \right\| \right).
\]
Further,
\[
\P\!\left( \Gam(d, 1) \leq \bon - \left\| \gov_{n, m, \ga} \right\| \right)
\geq \P\!\left( \Gam(d, 1) \leq \bon - b_n - d 2^{-m} \right),
\]
so that, as $m \to \infty$, uniformly in $\ga \in I_{n, m}$ we have
\begin{equation}
\label{P7}
\P\!\left( \Gam(d, 1) \leq \bon - \left\| \gov_{n, m, \ga} \right\| \right)
\geq (1 + o(1)) \P(\Gam(d, 1) < \bon - b_n);
\end{equation}
and, by \eqref{LB1}--\eqref{Theta}, the probability on the right in~\eqref{P7}, call it $\theta_n$, satisfies
\[
\theta_n = 1 - O\!\left( (\rL n)^{- 2 (d - 1)} (\rL_2 n)^{d - 2} \right) = 1 - o(1)\mbox{\ uniformly for $|a| \leq a_n$}
\]
as $n \to \infty$.

It is then clear from a Riemann-sum approximation that
\begin{align}
\lefteqn{\limsup_{m \to \infty} \fb_1(n, m)} \nonumber \\ 
&\leq n^2 \int_{\x \in A_n} \int_{\y \in A_n:\,\| \x \vee \y \| < \bon} e^{- (\| \x \| + \| \y \|)}
\exp\!\left[- (1 - \theta_n) n e^{- (\| \x \| + \| \y \|)} \right] \dd \y \dd \x. \label{P8}
\end{align}
In~\eqref{P8}, the integrand factor
\[
\exp\!\left[\theta_n n e^{- (\| \x \| + \| \y \|)} \right]
\]
is bounded above by
\[
\exp\!\left[\theta_n n e^{- 2 \bun} \right]
= \exp\!\left[O\!\left( n^{-1} (\rL_2 n)^{(1 + o(1)) 2} \right) \right] = 1 + o(1),
\]
so
\begin{equation}
\label{fb1nmJn}
\limsup_{m \to \infty} \fb_1(n, m) \leq (1 + o(1)) J_n\mbox{\ uniformly for $|a| \leq a_n$ as $n \to \infty$},
\end{equation}
with
\begin{equation}
\label{Jn}
J_n := 
n^2 \int_{\x \in A_n} \int_{\y \in A_n:\,\| \x \vee \y \| < \bon} e^{- (\| \x \| + \| \y \|)}
\exp\!\left[- n e^{- (\| \x \| + \| \y \|)} \right] \dd \y \dd \x. 
\end{equation}
So it remains to prove that
\begin{equation}
\label{JnO}
J_n = O((\rL n)^{- (d - 1)} (\rL_2 n)^d (\rL_3 n)^{-2})\mbox{\ uniformly for $|a| \leq a_n$}.
\end{equation} 

To assess the magnitude of $J_n$, we begin by making the change of variables from $\x = (x_1, \ldots, x_d)$ to 
$(\uu, z) \equiv (u_1, \ldots, u_{d - 1}, z)$ using
\[
x_j = u_j z\mbox{\ for $j \in [d - 1]$ and\ \ }x_d = [1 - (u_1 + \cdots +u_{d - 1})] z
\]
and a similar change of variables from $\y$ to $(\vv, w)$ to find
\[
J_n = \int\!z^{d - 1} n e^{-z} \exp(- n e^{-z}) w^{d - 1} n e^{-w} \exp(- n e^{-w}) \dd \vv \dd \uu \dd w \dd z,
\]
where the integral is over $(z, w, \uu, \vv)$ satisfying the constraints
\begin{align*}
&{} \bun < z < b_n,\ \ \bun < w < b_n, \\
&{} \uu \succ {\bf 0},\ \ \| \uu \| < 1,\ \ \vv \succ {\bf 0},\ \ \| \vv \| < 1, \\
&{} \sum_{j = 1}^d [(u_j z) \vee (v_j w)] < \bon,
\end{align*}
with
\[
u_d := 1 - (u_1 + \cdots + u_{d - 1})\mbox{\ \ and\ \ }v_d := 1 - (v_1 + \cdots + v_{d - 1}).
\]
Now make the change of variables from $(z, w)$ to $(s, t)$ using $s = z - \rL n$ and $t = w - \rL n$ to obtain
\begin{equation}
\label{Jnstuv}
J_n = \int\!(\rL n + s)^{d - 1} e^{-s} \exp(- e^{-s}) (\rL n + t)^{d - 1} e^{-t} \exp(- e^{-t}) \dd \vv \dd \uu \dd t \dd s,
\end{equation}
where now the integral is over $(s, t, \uu, \vv)$ satisfying the constraints
\begin{align}
&{} \bun - \rL n < s < b_n - \rL n,\ \ \bun - \rL n < t < b_n - \rL n, \label{st} \\
&{} \uu \succ {\bf 0},\ \ \| \uu \| < 1,\ \ \vv \succ {\bf 0},\ \ \| \vv \| < 1, \label{uv} \\
&{} \sum_{j = 1}^d [(u_j (\rL n + s)) \vee (v_j (\rL n + t))] < \bon \label{vee}.
\end{align}
From~\eqref{Jnstuv}--\eqref{vee}, we have (for large~$n$) the bounds
\begin{align}
J_n 
&\leq b_n^{2 (d - 1)} \int\!e^{-s} \exp(- e^{-s}) e^{-t} \exp(- e^{-t}) \dd \vv \dd \uu \dd t \dd s \label{Jnbnstuv} \\
&= b_n^{2 (d - 1)} [(\rL n)^{- c_n} - (\rL n)^{- \go_n}]^2 \int\!\dd \vv \dd \uu \label{Jnuv1} \\
&\leq (\rL n)^{2 (d - 1 - c_n)} \int\!\dd \vv \dd \uu 
= e^{2 (d - 1) a} \int\!\dd \vv \dd \uu \label{Jnuv2}, 
\end{align}
with the integral in~\eqref{Jnbnstuv} over $(s, t, \uu, \vv)$ subject to the constraints \eqref{st}--\eqref{uv} and the constraint
\begin{equation}
\label{uveev}
\sum_{j = 1}^d (u_j \vee v_j) < 1 + \eps_n := \frac{\bon}{b_n},
\end{equation}
and the integrals in~\eqref{Jnuv1}--\eqref{Jnuv2} over $(\uu, \vv)$ subject to the constraints \eqref{uv} and~\eqref{uveev}.
For $a \leq a_n$, the factor $e^{2 (d - 1) a}$ multiplying the integral on the right in~\eqref{Jnuv2} is bounded by 
$O((\rL_2 n) (\rL_3 n)^{-2})$.
So now it remains to show that the integral $q_n = \int\!\dd \vv \dd \uu$ appearing thrice in~\eqref{Jnuv1}--\eqref{Jnuv2} satisfies
\begin{equation}
\label{qnO}
q_n = O\!\left( \left( \frac{\rL_2 n}{\rL n} \right)^{d - 1} \right)\mbox{\ uniformly for $|a| \leq a_n$}.
\end{equation}

To prove~\eqref{qnO}, let $\U = (U_1, \ldots, U_d)$ and $\V = (V_1, \ldots, V_d)$ be independent random vectors each distributed Dirichlet$(1, \ldots, 1)$, \ie,\ uniformly distributed on $\cS_{d - 1}$.  Then, letting $x^+ = \max\{x, 0\}$ denote the positive part of $x \in \R$, a bit of thought (noting especially that almost surely $\U, \V \succ 0$ and $\| \U \| = 1 = \| \V \|$) establishes the inequality in
\[
q_n 
= \P\!\left( \sum_{j = 1}^d (U_j \vee V_j) < 1 + \eps_n \right)
= \P\!\left( \sum_{j = 1}^d (V_j - U_j)^+ < \eps_n \right)
\leq \sum_{k = 1}^d \binom{d}{k} q_{n, k}
\]
with
\[
q_{n, k} := \P\!\left( 0 < V_j - U_j < \eps_n\mbox{\ for $j \in [k]$ and\ }0 < U_j - V_j < \eps_n\mbox{\ for $j \in [d] / [k]$} \right).
\]
Writing $q_{n, k}$ naturally as a $(2 (d - 1))$-dimensional integral with variables $u_1, \ldots, u_{d - 1}, v_1, \ldots, v_{d - 1}$ and then making the change of variables
\[
v_j \mapsto x_j = v_j - u_j\mbox{\ for $j \in [k]$ and\ }u_j \mapsto x_j = u_j - v_j\mbox{\ for $j \in [d - 1] \setminus [k]$},
\]
we find that
\begin{equation}
\label{qnkint}
q_{n, k} = \int\!\dd x_1 \cdots \ddx x_{d - 1} \dd u_1 \cdots \ddx u_k \dd v_{k + 1} \cdots \ddx v_{d - 1},
\end{equation}
where the integral in~\eqref{qnkint} is over variables $(u_1, \ldots, u_k)$ and $(v_{k + 1}, \ldots, v_{d - 1})$ and 
$(x_1, \ldots, x_{d - 1})$ satisfying the following relations:
\begin{align*}
&{} 0 < x_j < \eps_n\mbox{\ for $j \in [d - 1]$}, \\
&{} 0 < u_j < 1 - x_j\mbox{\ for $j \in [k]$}, \\
&{} 0 < v_j < 1 - x_j\mbox{\ for $j \in [d - 1] \setminus [k]$}, \\
&{} \sum_{j = 1}^k u_j + \sum_{j = k + 1}^{d - 1} (v_j + x_j) < 1, \\
&{} \sum_{j = 1}^k (u_j + x_j) + \sum_{j = k + 1}^{d - 1} v_j < 1.
\end{align*}
By relaxing these constraints to
\begin{align*}
&{} 0 < x_j < \eps_n\mbox{\ for $j \in [d - 1]$}, \\
&{} 0 < u_j < 1\mbox{\ for $j \in [k]$}, \\
&{} 0 < v_j < 1\mbox{\ for $j \in [d - 1] \setminus [k]$}, \\
&{} \sum_{j = 1}^k u_j + \sum_{j = k + 1}^{d - 1} v_j < 1,
\end{align*}
we find
\[
q_{n, k} \leq \eps_n^{d - 1}.
\]
Since, by the definition of~$\eps_n$ at~\eqref{uveev},
\[
\eps_n = \frac{\bon - b_n}{b_n} \sim \frac{2 (d - 1) \rL_2 n}{\rL n}\mbox{\ uniformly for $|a| \leq a_n$},
\]
at last we conclude (again uniformly for $|a| \leq a_n$)
\[
q_n \leq (2^d - 2) \eps_n^{d - 1} \sim [2 (d - 1)]^{d - 1} (2^d - 2) \left( \frac{\rL_2 n}{\rL n} \right)^{d - 1}, 
\]
establishing~\eqref{qnO} and thereby completing the proof of the lemma.
\end{proof}

\section{Distribution of the smallest maximum}
\label{S:location}

\begin{remark}
\label{R:independence}
(a)~As already discussed in \refS{S:intro}, if $\Y_n$ denotes a ``generic maximum'' at epoch~$n$, then $\|\Y_n\|$ and 
$\U_n := \|\Y_n\|^{-1} \Y_n$ are independent, and $\U_n$ is uniformly distributed on $\cS_{d - 1}$.

(b)~Likewise, if the \emph{largest maximum} $\bm{\gl}_n$ is defined as the (almost surely unique) maximum at epoch~$n$ with the largest $\ell^1$-norm, then it easy to check that $\|\bm{\gl}_n\|$ [$ \equiv F^+_n$ in the notation of \cite{Fillboundary_2020} and \cite{FNS_2024} and~\eqref{F+n}] and $\V_n := \|\bm{\gl}_n\|^{-1} \bm{\gl}_n$ are independent, and $\V_n$ is uniformly distributed on 
$\cS_{d - 1}$.

(c)~There is no such simple result (for finite~$n$) concerning the \emph{smallest maximum} $\bm{\gs}_n$, defined as the (almost surely unique) maximum at epoch~$n$ with the smallest $\ell^1$-norm (equal to $\vp_n$).  Indeed, using the fact that $\| \X \|$ is distributed $\Gam(d, 1)$, it is not hard to see that $\bm{\gs}_2$ has exactly the density
\begin{align*}
\P(\bm{\gs}_2 \in \ddx \s)
&= 2\,\P(\X \in \ddx \s)\,\big[\P(\X \prec \s) + \P(\| \X \| > \| \s \|) - \P(\X \succ \s)\big] \\ 
&= 2 e^{-\| \s \|} \left[ \prod_{j = 1}^d (1 - e^{- s_j}) + e^{-\| \s \|} \sum_{j = 1}^{d - 1} \frac{\| \s \|^j}{j!} \right] \dd \s
\end{align*}
for $\s \in (0, \infty)^d$.  The continuous density here is not a function of $\| \s \|$, and so $\vp_2 = \|\bm{\gs}_2\|$ and 
$\|\bm{\gs}_2\|^{-1} \bm{\gs}_2$ are not independent.
\end{remark}

Notwithstanding \refR{R:independence}(c), we make the following conjecture.

\begin{conj}
\label{Conj:asy indep}
With $\bm{\gs}_n$ defined to be the (almost surely unique) maximum of $\{ \X^{(1)}, \ldots, \X^{(n)} \}$ with minimum 
$\ell^1$-norm, $\vp_n$ and $\vp_n^{-1} \bm{\gs}_n$ are asymptotically independent, and $\vp_n^{-1} \bm{\gs}_n$ converges in distribution to the uniform distribution on $\cS_{d - 1}$. 
\end{conj}

In light of the heuristic argument spanning \eqref{Enx}--\eqref{parameter} in \refS{S:intro}, it is quite believable that, for each fixed $a \in \R$ and for large~$n$, the point process of vectors $\x \in (0, \infty)^d$ with $\bun < \| \x \| \leq \bzn$ that are maxima of $\{ \X^{(1)}, \ldots, \X^{(n)} \}$ at epoch~$n$ is, in some suitable sense, well approximated by a Poisson point process, call it $\cP_n$, with intensity measure
\[
\nu_n(\ddx \x) := n e^{- \| \x \|} \exp(- n e^{- \| \x \|}), \quad \bun < \| \x \| \leq \bzn.
\] 
Let $H_n$ denote the event that $\cP_n((0, \infty)^d) \geq 1$, and let $\U$ be independent of $\cP_n$ and uniformly distributed over $\cS_{d - 1}$.
Over the event $H_n$, let $\hgs_n$ denote the (almost surely unique) point of $\cP_n$ with minimum $\ell^1$-norm, say $\hvp_n$; and over $H_n^c$, let $\hvp_n := 1$ and $\hgs_n := \U$.  Then, for Borel $A \subseteq \cS_{d - 1}$, we would expect to have
\begin{align}
\P(\vp_n \leq \bzn,\,\vp_n^{-1} \bm{\gs}_n \in A)
&\approx \P(H_n \cap \{\hvp_n^{-1} \hgs_n \in A\}) \label{intersection} \\
&= \P(H_n)\,\P(\U \in A) \label{prod} 
\end{align}
[indeed, we know that the equality of the right side of~\eqref{intersection} and the expression in~\eqref{prod} is true exactly!]\ and, in particular,
\begin{equation}
\label{PHn}
\P(\vp_n \leq \bzn) \approx \P(H_n).
\end{equation}
Combining \eqref{prod}--\eqref{PHn} would then give \refConj{Conj:asy indep}.  Perhaps the point-process approximation by $\cP_n$ we have surmised here can be established rigorously using the techniques of 
\cite[Chapter~10]{Barbour_1992}, which include \cite[Theorem~2]{Arratia_1989}.

\appendix


\section{Collected results from Fill et al.~\cite{FNS_2024}, and closely related results}
\label{A:FNS}

For the reader's convenience, this appendix collects results from~\cite{FNS_2024}, and closely related results, used in the present work.
\medskip

Our first lemma is used in the proof of \refP{P:BE again 2}.

\begin{lemma}[\cite{FNS_2024}, Propositions~2.3 and~3.2]
\label{L:moment method}
If $\Bigl( \tb_n \Bigr)$ is any sequence satisfying
\begin{equation}
\label{bndefhat}
\tb_n \coloneq \rL n - \rL_3 n - \rL \tc_n\mbox{\rm \ with $\tc_n > 0$ and $\tc_n = \Theta(1)$},
\end{equation}
then as $n \to \infty$ we have
\begin{align*}
\P\!\left( \vp_n \leq \tb_n \right) 
&\leq \E \rho_n\!\left( \tb_n \right) 
= (1 + o(1))\,\frac{1}{(d - 1)!} (\rL n)^{d - 1 - \tc_n}
\end{align*}
and
\begin{IEEEeqnarray*}{+rCl+x*}
\P\!\left( \vp_n \geq \tb_n \right) 
&\leq& (1 + o(1))\,(d - 1)! (\rL n)^{ - (d - 1 - \tc_n)}. &\qed
\end{IEEEeqnarray*}
\end{lemma}
\smallskip

Define
\begin{equation}
\label{Jjdef}
J_j(x) \coloneq \int_x^{\infty} (\rL z)^j e^{-z} \dd z
\end{equation}
and note \cite[(A.3)]{FNS_2024} that
\begin{equation}
\label{Jjasy}
J_j(x) \sim (\rL x)^j e^{-x}\mbox{\ as $x \to \infty$}. 
\end{equation}
The next lemma is used in the proof of~\eqref{TV11}.
\begin{lemma}
\label{L:Erhonbun}
We have
\[
\E \rho_n(\bun) = O((\rL n)^{d - 1 - \go_n}) = o((\rL n)^{- (d - 1)}).
\]
\end{lemma}

\begin{proof}
To bound $\E \rho_n(\bun)$, follow along the calculations in the proof of Lemma 2.1 in \cite[Appendix~A]{FNS_2024}.  One finds (compare \cite[(A.6)]{FNS_2024})
\begin{equation}
\label{Erhonbun bound}
\E \rho_n(\bun)
\leq \frac{1 + (n - 1)^{-1}}{(d - 1)!} \sum_{j = 0}^{d - 1} (-1)^j \binom{d - 1}{j} (\rL n)^{d - 1 - j}
[J_j((n - 1) e^{- \bun}) - J_j(n)],
\end{equation}
with, by~\eqref{Jjasy}, $J_j(n) \sim (\rL n)^j e^{-n}$ and, using the growth assumption $\rL \go_n = o(\rL_3 n)$ at~\eqref{simp1} and~\eqref{simp2},
\begin{align}
J_j((n - 1) e^{- \bun})
&\sim [\rL(n - 1) - \bun]^j \exp[- (n - 1) e^{- \bun}] \nonumber \\
&\sim [\rL n - \bun]^j \exp[- (n - 1) e^{- \bun}] \nonumber \\
&= [\rL_3 n + \rL \go_n]^j \exp[- (n - 1) n^{-1} (\rL_2 n) \go_n] \nonumber \\
&\sim (\rL_3 n)^j \exp[- (n - 1) n^{-1} (\rL_2 n) \go_n] \label{simp1} \\
&\sim (\rL_3 n)^j (\rL n)^{- \go_n} \label{simp2}.
\end{align}
Thus the $j = 0$ term in~\eqref{Erhonbun bound} predominates and yields the desired result.
\end{proof}

The next lemma is used in the proof of \refP{P:TV33}.

\begin{lemma} 
\label{L:PEn}
Let $E_n$ be the event that a Poisson process with intensity measure
\[
\mu_n(\ddx \x) := n e^{- \| \x \|} {\bf 1}(\x \succ {\bf 0}) \dd \x
\]
has a maximum with $\ell^1$-norm larger than $\bon = \rL n + 2 (d - 1) \rL_2 n$.  Then
\[
\P(E_n) = O((\rL n)^{ - (d - 1)}).
\]
\end{lemma}

\begin{proof}
The number of maxima satisfying the $\ell^1$-norm condition is of course simply the number of Poisson points satisfying the same condition, and so is distributed Poisson with parameter
\begin{align*}
\int_{\x \succ {\bf 0}:\| \x \| > \bon} \mu_n(\ddx \x)
&= n\,\P(\Gam(d, 1) > \bon)
= n\,\P(\Po(\bon) < d) \\
&\sim n\,\P(\Po(\bon) = d - 1) \sim \frac{1}{(d - 1)!} (\rL n)^{- (d - 1)}. 
\end{align*}
Thus
\begin{align*}
\P(E_n) 
&= 1 - \exp\left[ - (1 + o(1)) \frac{1}{(d - 1)!} (\rL n)^{- (d - 1)} \right] \\
&\sim \frac{1}{(d - 1)!} (\rL n)^{- (d - 1)} = O((\rL n)^{ - (d - 1)}),
\end{align*}
as claimed.
\end{proof}

The next lemma is used in the proof of~\eqref{TV55}.  Given the convergence of 
$\Var \rho_n(\bzn) - \E \rho_n(\bzn)$ to~$0$ established for fixed~$a$ by~\eqref{mean} and \cite[Lemma 3.1 and Remark B.1]{FNS_2024}, it is not surprising that 
$\Var \Non - \E \Non = o(1)$, but it is nontrivial to prove this (with suitable uniformity in~$a$) and to further quantify $o(1)$.

\begin{lemma}
\label{L:VarNon}
As $n \to \infty$, uniformly for $|a| \leq a_n$ we have
\begin{equation}
\label{VarNon again2}
\Var \Non - \E \Non = O\!\left( (\rL_2 n)^{ - \left( d - \tfrac32 \right)} (\rL_3 n)^{-1} \right).
\end{equation}
\end{lemma}

\begin{proof}
To treat $\Var \Non$ we follow the approach to $\Var \rho_n(b_n)$ in~\cite{FNS_2024} and start with the second factorial moment of $\Non$.  Letting $\x \parallel \y$ denote incomparability of~$\x$ and~$\y$ with respect to $\prec$ (more precisely, that neither~$\x$ weakly dominates~$\y$ coordinate-by-coordinate nor vice versa), we have, by the Mecke equation \cite[Theorem~4.4]{Last_2018},
\begin{align}
\lefteqn{\hspace{-.1in}\E \left[ \Non \left( \Non - 1 \right) \right]} \nonumber \\
&= n^2 \int_{\x, \y \in A_n:\,\x \parallel \y} e^{- (\| \x \| + \| \y \|)} 
\P\!\left( \Pion(O^+_{\x}) = 0\mbox{\ and\  }\Pion(O^+_{\y}) = 0 \right) \dd \x \dd \y \nonumber \\
&= \left( \E \Non \right)^2 - I_{n, 0} + I_n, \label{2fac}
\end{align}
where
\begin{align*}
I_{n, 0} 
&:= n^2 \int_{\x, \y \in A_n:\,\x \not\parallel \y} e^{- (\| \x \| + \| \y \|)} 
\P\!\left( \Pion(O^+_{\x}) = 0 \right) \P\!\left( \Pion(O^+_{\y}) = 0 \right) \dd \x \dd \y, \\
I_n 
&:= n^2 \int_{\x, \y \in A_n:\,\x \parallel \y} e^{- (\| \x \| + \| \y \|)} 
\left[ \P\!\left( \Pion(O^+_{\x}) = 0\mbox{\ and\  }\Pion(O^+_{\y}) = 0 \right) \right. \\ 
&{} \qquad \qquad \qquad \qquad \qquad \qquad \ 
- \left. \P\!\left( \Pion(O^+_{\x}) = 0 \right) \P\!\left( \Pion(O^+_{\y}) = 0 \right) \right] \dd \x \dd \y.
\end{align*}
Clearly $I_{n, 0} \geq 0$, and it is easy to see that the integrand factor in square brackets for $I_n$ is nonnegative, so $I_n \geq 0$, too.
From~\eqref{2fac} it follows that
\[
\Var \Non - \E \Non 
= \E \left[ \Non \left( \Non - 1 \right) \right] - \left( \E \Non \right)^2
= - I_{n, 0} + I_n. 
\]
We will complete the proof of \refL{L:VarNon} by proving that
\begin{align}
I_{n, 0} 
&= O\!\left( (\rL n)^{- (d - 1)} (\rL_2 n)^2 (\rL_3 n)^{-2} \right) \label{In0} \\ 
&= o\!\left( (\rL_2 n)^{ - \left( d - \tfrac32 \right)} (\rL_3 n)^{-1} \right),\mbox{\ uniformly for $|a| \leq a_n$} \nonumber
\end{align}
and
\begin{equation}
\label{In}
I_n = O\!\left( (\rL_2 n)^{ - \left( d - \tfrac32 \right)} (\rL_3 n)^{-1} \right)\mbox{\ uniformly for $|a| \leq a_n$}.
\end{equation}

To prove~\eqref{In0}, we begin by recalling from~\eqref{PPionO+x=0 2} that for $\x \in A_n$ we have
\begin{equation}
\label{P5}
\P\!\left( \Pion(O^+_{\x}) = 0 \right)
= \exp\!\left[ - n e^{- \| \x \|} (1 - u_n(\| \x \|)) \right]\mbox{\ for $\x \in A_n$}
\end{equation}
with
\begin{align*}
u_n(w) 
&\equiv u_{n, a}(w) \\ 
&= O\!\left( (\rL n)^{- 2 (d - 1)} (\rL_2 n)^{d - 2} \right)\mbox{\ uniformly for $w \in (\bun, b_n]$ and $|a| \leq a_n$},
\end{align*}
so 
that
\begin{align*}
I_{n, 0} 
&= 2 n^2 \int_{\x, \y \in A_n:\,\x \prec \y} e^{- (\| \x \| + \| \y \|)} \exp\!\left[ - n e^{- \| \x \|} (1 - u_n(\| \x \|)) \right] \\ 
&{} \qquad \qquad \qquad \qquad {} \times 
\exp\!\left[ - n e^{- \| \y \|} (1 - u_n(\| \y \|)) \right] \dd \x \dd \y \\
&= 2 n^2 \int_{\x \in A_n} \int_{\z \succ {\bf 0}:\,\| \z \| \leq b_n - \| \x \|} 
e^{- \| \x \|} e^{- (\| \x \|) + \| \z \|)} \exp\!\left[ - n e^{- \| \x \|} (1 - u_n(\| \x \|)) \right] \\ 
&{} \qquad \qquad \qquad \qquad \qquad {} \times
\exp\!\left[ - n e^{- (\| \x \| + \| \z \|)} (1 - u_n(\| \x \| + \| \z \|)) \right] \dd \z \dd \x \\
&= 2 n^2 \int_{\bun}^{b_n} \int_0^{b_n - x} 
\frac{x^{d - 1}}{(d - 1)!} e^{- 2 x} \frac{\gz^{d - 1}}{(d - 1)!} e^{- \gz} \exp\!\left[ - n e^{- x} (1 - u_n(x)) \right] \\ 
&{} \qquad \qquad \qquad \qquad {} \times
\exp\!\left[ - n e^{- x} e^{- \gz} (1 - u_n(x + \gz)) \right] \dd \gz \dd x
\end{align*}
From here, change variables using $x = \rL n - \rL y$ and $\gz = - \rL z$ to obtain, with $\han$ defined at~\eqref{han} (we could do slightly but immaterially better by discarding the absolute value signs there),
\begin{align*}
I_{n, 0} 
&= \frac{2}{[(d - 1)!]^2} \int_{\gb_n}^{\gbun} \int_{\gb_n / y}^1
(\rL n - \rL y)^{d - 1} y\,( - \rL z)^{d - 1} \\ 
&{} {} \times
\exp\!\left[ - y (1 - u_n(\rL n - \rL y)) \right] \exp\!\left[ - y z (1 - u_n(\rL n - \rL y - \rL z)) \right] \dd z \dd y \\
&\leq \frac{2 (\rL n)^{d - 1} \exp\!\left[ - (1 - \han) \gb_n \right]}{[(d - 1)!]^2} 
\!\!\int_{\gb_n}^{\gbun}\!\!y \exp\!\left[ - (1 - \han) y \right] \int_{\gb_n / y}^1\!\!\!\!( - \rL z)^{d - 1} \dd z \dd y.
\end{align*}

But for $y \in [\gb_n, \infty)$ we have
\[
\int_{\gb_n / y}^1\!( - \rL z)^{d - 1} \dd z
= \int_0^{\rL y - \rL \gb_n}\!w^{d - 1} e^{-w} \dd w \leq (d - 1)!, 
\]
so, using~\eqref{uny} at the asymptotic equivalence to follow, uniformly for $|a| \leq a_n$ we have
\begin{align*}
I_{n, 0}
&\leq \frac{2 (\rL n)^{d - 1} \exp\!\left[ - (1 - \han) \gb_n \right]}{(d - 1)!} 
\int_{\gb_n}^{\infty}\!\!y \exp\!\left[ - (1 - \han) y \right] \dd y \\
&\sim \frac{2}{(d - 1)!} (\rL n)^{d - 1} \gb_n e^{- 2 \gb_n}
= O\!\left( (\rL n)^{- (d - 1)} (\rL_2 n)^2 (\rL_3 n)^{-2} \right),
\end{align*} 
and~\eqref{In0} is established.

Now we turn our attention to $I_n$ and the proof of~\eqref{In}.  First note that for any~$\x, \y \in A_n$ we have
\begin{align*}
\P\!\left( \Pion(O^+_{\y}) = 0\,|\,\Pion(O^+_{\x}) = 0 \right)
&= \P\!\left( \Pion(O^+_{\y}) = 0\,|\,\Pion(O^+_{\x \vee \y}) = 0 \right) \\
&= \frac{\P\!\left( \Pion(O^+_{\y}) = 0 \right)}{\P\!\left( \Pion(O^+_{\x \vee \y}) = 0 \right)};
\end{align*}
thus
\begin{align*}
I_n 
&= n^2 \int_{\x, \y \in A_n:\,\x \parallel \y} e^{- (\| \x \| + \| \y \|)}
\P\!\left( \Pion(O^+_{\x}) = 0 \right) \P\!\left( \Pion(O^+_{\y}) = 0 \right) \\
&{} \qquad \qquad \qquad \qquad {} \times \left[ \frac{1}{\P\!\left( \Pion(O^+_{\x \vee \y}) = 0 \right)} - 1 \right] \dd \x \dd \y.
\end{align*}
For the first two of the three probabilities appearing in this last expression we use~\eqref{P5}; for the third, we note simply that
\[
\P\!\left( \Pion(O^+_{\x \vee \y}) = 0 \right) \geq \exp(- n e^{- \| \x \vee \y \|}).
\]
We then find
\begin{align*}
I_n 
&\leq n^2 \int_{\x, \y \in A_n:\,\x \parallel \y} e^{- (\| \x \| + \| \y \|)}
\exp\!\left[ - n e^{- \| \x \|} (1 - u_n(\| \x \|)) \right] \\
&{} \qquad \qquad \qquad \quad {} \times \exp\!\left[ - n e^{- \| \y \|} (1 - u_n(\| \y \|)) \right]
\left[ \exp(n e^{- \| \x \vee \y \|}) - 1 \right] \dd \x \dd \y.
\end{align*}
Here, the integrand factors $\exp\!\left[ n e^{- \| \x \|} u_n(\| \x \|) \right]$ and $\exp\!\left[ n e^{- \| \y \|} u_n(\| \y \|) \right]$ are each bounded by 
\[
\exp\!\left( \han \gbun \right) = \exp\!\left[ O\!\left( (\rL n)^{- 2 (d - 1)} (\rL_2 n)^{d - 1} \go_n \right) \right] = 1 + o(1). 
\]
Therefore $I_n$ is bounded by $(1 + o(1))$ times
\begin{align}
\lefteqn{\hspace{.25in}n^2 \int_{\x, \y \in A_n:\,\x \parallel \y} e^{- (\| \x \| + \| \y \|)}
\exp\!\left( - n e^{- \| \x \|} \right)} \nonumber \\
&{} \qquad \qquad \qquad \quad {} \times \exp\!\left( - n e^{- \| \y \|} \right)
\left[ \exp(n e^{- \| \x \vee \y \|}) - 1 \right] \dd \x \dd \y \nonumber \\
&\leq n^2 \int_{\x, \y \succ 0:\,\x \parallel \y,\,\| \x \| \leq b_n,\,\| \y \| \leq b_n} e^{- (\| \x \| + \| \y \|)}
\exp\!\left( - n e^{- \| \x \|} \right) \nonumber \\
&{} \qquad \qquad \qquad \qquad \qquad \qquad \quad {} \times \exp\!\left( - n e^{- \| \y \|} \right)
\left[ \exp(n e^{- \| \x \vee \y \|}) - 1 \right] \dd \x \dd \y. \label{In bound}
\end{align}

After a simple change of variables, \eqref{In bound} is seen to agree with the expression $I_n$ appearing at 
\cite[(B.3)]{FNS_2024}, except that the factor $(n + 2) (n + 1)$ appearing there becomes $n^2$ at~\eqref{In bound} and the restrictions $\x_{\times} \geq e^{- b_{n + 2}}$ and $\y_{\times} \geq e^{- b_{n + 2}}$ on the integral at \cite[(B.3)]{FNS_2024} become our $\x_{\times} \geq e^{- b_n}$ and $\y_{\times} \geq e^{- b_n}$.  
(Here $\z_{\times} := \prod_{j = 1}^d z_j$ for any vector~$\z$.)
We can decompose our $I_n$ as 
$\sum_{k = 1}^{d - 1} \binom{d}{k} I_{n, k}$ in the same fashion as done at \cite[(B.7)]{FNS_2024}, and we similarly find, uniformly for $|a| \leq a_n$, that
\[
I_{n, k} 
= O((\rL n)^{d - 1} \gb_n^{- (d - 1)} e^{- \gb_n}) \\
= O\!\left( (\rL_2 n)^{ - \left( d - \tfrac32 \right)} (\rL_3 n)^{-1} \right).
\]
This establishes~\eqref{In} and completes the proof of the lemma.
\end{proof}



\bibliographystyle{elsarticle-num} 
\bibliography{records.bib}

\begin{thebibliography}{10}
\expandafter\ifx\csname url\endcsname\relax
  \def\url#1{\texttt{#1}}\fi
\expandafter\ifx\csname urlprefix\endcsname\relax\def\urlprefix{URL }\fi
\expandafter\ifx\csname href\endcsname\relax
  \def\href#1#2{#2} \def\path#1{#1}\fi

\bibitem{FNS_2024}
J.~A. Fill, D.~Q. Naiman, A.~Sun,
  \href{https://drops.dagstuhl.de/entities/document/10.4230/LIPIcs.AofA.2024.28}{{Sharpened
  localization of the trailing point of the {P}areto record frontier}}, in:
  C.~Mailler, S.~Wild (Eds.), 35th International Conference on Probabilistic,
  Combinatorial and Asymptotic Methods for the Analysis of Algorithms (AofA
  2024), Vol. 302 of Leibniz International Proceedings in Informatics (LIPIcs),
  Schloss Dagstuhl -- Leibniz-Zentrum f{\"u}r Informatik, Dagstuhl, Germany,
  2024, pp. 28:1--28:21.
\newblock \href {https://doi.org/10.4230/LIPIcs.AofA.2024.28}
  {\path{doi:10.4230/LIPIcs.AofA.2024.28}}.
\newline\urlprefix\url{https://drops.dagstuhl.de/entities/document/10.4230/LIPIcs.AofA.2024.28}

\bibitem{Arnold_1998}
B.~C. Arnold, N.~Balakrishnan, H.~N. Nagaraja,
  \href{https://doi.org/10.1002/9781118150412}{Records}, Wiley Series in
  Probability and Statistics: Probability and Statistics, John Wiley \& Sons,
  Inc., New York, 1998, a Wiley-Interscience Publication.
\newblock \href {https://doi.org/10.1002/9781118150412}
  {\path{doi:10.1002/9781118150412}}.
\newline\urlprefix\url{https://doi.org/10.1002/9781118150412}

\bibitem{Fillboundary_2020}
J.~A. Fill, D.~Q. Naiman, \href{https://doi.org/10.1214/20-ejp492}{The {P}areto
  record frontier}, Electron. J. Probab. 25 (2020) Paper No. 92, 24 pages.
\newblock \href {https://doi.org/10.1214/20-ejp492}
  {\path{doi:10.1214/20-ejp492}}.
\newline\urlprefix\url{https://doi.org/10.1214/20-ejp492}

\bibitem{Fillmulti_2023}
J.~A. Fill, Breaking multivariate records, Electron. J. Probab. 28 (2023)
  1--27.
\newblock \href {https://doi.org/10.1214/23-EJP968}
  {\path{doi:10.1214/23-EJP968}}.

\bibitem{Bai_2005}
Z.-D. Bai, L.~Devroye, H.-K. Hwang, T.-H. Tsai,
  \href{https://doi.org/10.1002/rsa.20053}{Maxima in hypercubes}, Random
  Structures Algorithms 27~(3) (2005) 290--309.
\newblock \href {https://doi.org/10.1002/rsa.20053}
  {\path{doi:10.1002/rsa.20053}}.
\newline\urlprefix\url{https://doi.org/10.1002/rsa.20053}

\bibitem{Baryshnikov_2000}
Y.~Baryshnikov, \href{https://doi.org/10.1007/s004400050002}{Supporting-points
  processes and some of their applications}, Probab. Theory Related Fields
  117~(2) (2000) 163--182.
\newblock \href {https://doi.org/10.1007/s004400050002}
  {\path{doi:10.1007/s004400050002}}.
\newline\urlprefix\url{https://doi.org/10.1007/s004400050002}

\bibitem{Kiefer_1972}
J.~Kiefer, Iterated logarithm analogues for sample quantiles when
  {$p_n\downarrow 0$}, Proceedings of the {S}ixth {B}erkeley {S}ymposium on
  {M}athematical {S}tatistics and {P}robability ({U}niv. {C}alifornia,
  {B}erkeley, {C}alif., 1970/1971), {V}ol. {I}: {T}heory of statistics (1972)
  227--244.

\bibitem{Arratia_1989}
R.~Arratia, L.~Goldstein, L.~Gordon,
  \href{http://links.jstor.org/sici?sici=0091-1798(198901)17:1<9:TMSFPA>2.0.CO;2-X&origin=MSN}{Two
  moments suffice for {P}oisson approximations: {T}he {C}hen-{S}tein method},
  Ann. Probab. 17~(1) (1989) 9--25.
\newline\urlprefix\url{http://links.jstor.org/sici?sici=0091-1798(198901)17:1<9:TMSFPA>2.0.CO;2-X&origin=MSN}

\bibitem{Barbour_1992}
A.~D. Barbour, L.~Holst, S.~Janson, Poisson Approximation, Vol.~2 of Oxford
  Studies in Probability, The Clarendon Press, Oxford University Press, New
  York, 1992, {O}xford {S}cience {P}ublications.

\bibitem{Chung_2001}
K.~L. Chung, A Course in Probability Theory, third edition Edition, Academic
  Press, Inc., San Diego, CA, 2001.

\bibitem{Sun_2025}
A.~Sun, Studies in multivariate {P}areto records, Ph.D. thesis, The Johns
  Hopkins University (2025).

\bibitem{Resnick_2008}
S.~I. Resnick, Extreme Values, Regular Variation and Point Processes, Springer
  Series in Operations Research and Financial Engineering, Springer, New York,
  2008, reprint of the 1987 original.

\bibitem{BillingsleyPM_2012}
P.~Billingsley, Probability and Measure, anniversary edition Edition, Wiley
  Series in Probability and Statistics, John Wiley \& Sons, Inc., Hoboken, NJ,
  2012, with a foreword by Steve Lalley and a brief biography of Billingsley by
  Steve Koppes.

\bibitem{Last_2018}
G.~Last, M.~Penrose, Lectures on the {P}oisson Process, Vol.~7 of Institute of
  Mathematical Statistics Textbooks, Cambridge University Press, Cambridge,
  2018.

\end{thebibliography}





\end{document}